%% file: SingleFile.tex
\documentclass[11pt]{amsart}
\usepackage{geometry}                
\usepackage[parfill]{parskip}    
\usepackage{graphicx}
\usepackage{amssymb,amsfonts}
\usepackage{epstopdf}
\allowdisplaybreaks
\include{symb1}

\title{Spectra and Systems of Equations}
\author{Jason Bell, Stanley Burris, Karen Yeats}
\date{\today}                           

\begin{document}
\maketitle

 \begin{abstract}
 In a previous work we introduced an elementary method to analyze the periodicity 
 of a generating function defined by a single equation $y=G(x,y)$. This was based on 
 deriving a single set-equation $Y = \gam(Y)$ defining the spectrum of the generating
 function. This paper focuses on extending the analysis of 
 periodicity to generating functions defined by a system of equations $\by = \bG(x,\by)$. 

The final section looks at periodicity results for the spectra of monadic second-order 
classes whose spectrum is determined by an equational specification---an observation of
Compton shows that monadic-second order classes of trees have this property.
This section concludes with a substantial simplification of the proofs in the 2003 
foundational paper on spectra by Gurevich and Shelah \cite{GuSh2003}, namely
new proofs are given of:
 (1) every monadic second-order class of  $m$-colored functional digraphs 
 is eventually periodic, and (2) the monadic second-order theory of finite trees
 is decidable.
\end{abstract}

\section{Introduction}
   
Following Flajolet and  Sedgewick \cite{FlSe2009},  
a {\em combinatorial class} $\cA$ is a class of  objects with a function $||\ ||$ that assigns a
positive integer size to each object in the class, satisfying the condition that there are only
finitely many objects of each size. 
We deviate from the  definition in \cite{FlSe2009} by not having objects of  size  0.
Letting $a(n)$ be the number of objects of size $n$, one has the {\em generating function} 
$\displaystyle A(x) = \sum_{n=1}^\infty a(n) x^n$. 
   
\subsection{Generating Functions defined by Systems of Equations}
   
 Cayley \cite{Cayley1857} noted in his very first paper on trees\footnote
	 {Unless stated otherwise, all trees in this paper are assumed to be {\em rooted}.} 
 in 1857 that one has an equation
\[
   \sum_{n\ge 1} a(n)x^n\   =  \ x \cdot \prod_{n\ge 1} (1-x^n)^{-a(n)}
\]
which yields a recursive procedure to calculate the values of $a(n)$. 
Cayley used this to calculate the first 13 coefficients $a(n)$, that is, the numbers of 
such trees of sizes 1 through 13 (two of the numbers were not calculated correctly).\footnote
    {Actually, as noted in \cite{BBY2006}, Cayley's equation was slightly different since his 
    $a(n)$ counted the number of trees with $n$ edges, which is the number of trees
    with $n+1$ vertices.}
In 1937 P\'olya (see \cite{PoRe1987}) would  rewrite this equation as
\[
   A(x)\   =  \ x \cdot \exp\left(\sum_{m=1}^\infty A(x^m)/m\right),
\]
a form which could be viewed as a functional equation for $A(x)$, 
with important analytic properties based on the fact that  the radius 
of convergence $\rho$ of $A(x)$ is less than 1 (which is easily proved).
This allowed P\'olya to invoke the  implicit function theorem and results of 
Darboux to show that $A(x)$ has a square-root singularity at $\rho$, leading to 
the asymptotic form 
$C\rho^{-n}n^{-3/2}$ 
for the coefficients $a(n)$.\footnote
	{In \cite{BBY2006} we showed that a similar analysis can be carried out for most $A(x)$ 
	defined by a single non-linear equation $y = \Theta(y)$ where $\Theta$ is constructed 
	from the variables $x, y$, operations $+, \cdot ,\circ$, and unary operators that correspond 
	to (restrictions of) the standard constructions of Multiset, Sequence and (directed or 
	undirected) Cycle.} 
Many natural classes of trees are specified recursively by a single equation, for example 
planar binary trees, also known as $(0,2)$-trees, where the generating function
$A(x)$ solves the equation 
$y\   =  \ x\big(1 +  y^2 \big) $.

Although generating functions defined by a single equation cover many interesting 
cases, Example \ref{blue and red}, at the end of the introduction section, hints at 
the value of considering generating functions defined by a system of several equations. 

 \subsection{Spectra and Periodicity}
In 1952 the \textit{Journal of Symbolic Logic} initiated a section devoted to unsolved problems
 in the field of symbolic logic. The first problem, posed by Heinrich Scholz 
 \cite{Scholz1952},  was the following. Given a sentence $\varphi$ from first-order logic, 
he defined the {\em spectrum} of $\varphi$ to be the set of sizes of the finite models of $\varphi$. 
For example, binary trees can be defined by such a $\varphi$, and its spectrum is the arithmetical 
 progression $\{1,3,5,\ldots\}$. The algebraic structures called fields can also be defined 
 by such a $\varphi$, with the spectrum
  being the set $\{2,4,\ldots,3,9,\ldots\}$ of powers of prime numbers. The possibilities for the 
  spectrum of a first-order sentence are amazingly complex.\footnote
  {Asser's 1955 conjecture, that\textit{ the complement of a first-order spectrum is always going to
  be a first-order spectrum}, is still open --- it is known, through the work of Jones and Selman and Fagin in the 1970s, 
 that this conjecture is equivalent to the question of whether the complexity class NE of 
  problems decidable by a nondeterministic machine in exponential time is closed under complement.   
  Thus, Asser's conjecture is, in fact, one of the notoriously hard questions of computational 
  complexity theory. Stockmeyer \cite{Stockmeyer1987}, p.~33, states that if Asser's conjecture is false then 
  NP $\neq$ co-NP, and hence P $\neq$ NP.}
  
   Scholz's problem was to find a necessary and sufficient condition for a set $S$ of natural
    numbers to be the spectrum of some first-order sentence $\varphi$. This problem led to a 
    great deal of research by logicians on the topic of spectra --- see for example the recent
    survey paper \cite{DJMM2009} of Durand, Jones, Makowsky, and More.
    Periodicity is one of the properties that has
    been examined in the context of studying spectra. 

\begin{definition}  \label{defn periodic}
 $\bbN$ is the set of non-negative integers, $\bbP$ is the set of positive integers. 

For $A\subseteq \mathbb N$,
\begin{thlist} 

\item
 $A$ is {\em periodic} if there is a positive integer $p$ such that  $ p+A \subseteq A$, 
 that is, $a\in A$ implies $p+a\in A$. Such an integer $p$ is a {\em period} of $A$.
 
 \item
A is {\em eventually periodic} if there is a positive integer 
$p$ such that $ p+A$ is eventually in $A$, that is, there is an $m$ such that for $a\in A$, 
if $a\ge m$ then $p+a\in A$.  Such a $p$ is an {\em eventual period} of $A$.
  
\end{thlist} 
\end{definition} 

Clearly every arithmetical progression and every cofinite subset of $\bbN$ is periodic; and 
every periodic set is eventually periodic. Finite sets are eventually periodic. As will be seen,
periodicity seems to be a natural property for the spectra of combinatorial classes specified 
by a system of equations. The famous Skolem-Mahler-Leech Theorem (see, for example,
\cite{FlSe2009}, p.~266) says that the spectrum of every rational function $P(x)/Q(x)$
 in $\mathbb{Q}(x)$ is eventually periodic. 
 Consequently polynomial systems $\by = \bG(x,\by)$ with rational coefficients that are linear 
 in the variables $y_i$, and with a non-singular Jacobian matrix $\partial(\by - \bG)/\partial \by$,
  have power series solutions $y_i = T_i(x)$ such that the support sets of
  the coefficient sequences $t_i(n)$ of the  $T_i(x)$ are eventually periodic.
However, much simpler methods give this periodicity result for the non-negative
$\by$-linear systems considered in this paper.

If the spectrum $A$ of a combinatorial class $\cA$ is eventually periodic then one has
the possibility,  as in the case of regular languages and well-behaved irreducible systems, that
 the class $\cA$ decomposes into a finite subclass $\cA_0$, along with finitely many subclasses
 $\cA_i$, such that the spectrums $A_i$ are arithmetical progressions $a_i + b_i\cdot \bbN$, and 
 the generating functions $A_i(x)$ have well-behaved coefficients (e.g., monotone increasing)
 on $A_i$.\footnote{
 The comments in this paragraph are related to Question 7.4 in Compton's 1989 paper \cite{Co2} on monadic second order logical limit laws.
 }
 
 In the study of spectra of combinatorial classes,  logicians have dominated the
 literature thanks to powerful tools like Ehrenfeucht-Fra\"{i}ss\'{e} games. In this paper an 
 alternate approach to the spectra of combinatorial classes is developed using \textit{systems 
 of set-equations} derived directly from  specifications, or from systems of equations defining
generating functions. This method was briefly 
 introduced in 2006 in \cite{BBY2006}, to study the spectrum of a combinatorial class 
 defined by a single equation.  
For example, the class of planar binary trees is specified by the equation
$\cT = \{\bullet\}\, \cup\, \{\bullet\}/\Seq_2(\cT)$,
which one can read as: {\em the class of planar binary trees is the smallest class $\cT$ 
which has the one-element tree `$\bullet$' and is closed under taking any sequence of
two trees and adjoining a new root `$\bullet$' }.  From the specification equation the 
generating function $T(x)$ of $\cT$ satisfies $T(x) = x + x\cdot T(x)^2$, a simple 
quadratic equation that can be solved for $T(x)$. One also says that $T(x)$ is a solution
to the polynomial equation $y = x + x\cdot y^2$. For the spectrum $T$ of $\cT$ one
has the equation
$T = 1\, \cup\, (1 + 2\star T)$,
so $T$ satisfies the set-equation $Y = 1\,\cup\,(1+2\star Y)$. 
 (See $\S$\ref{SectSetOps} for the notation used here.)
Solving this set-equation gives the periodic spectrum 
$T=  1+ 2\cdot \bbN$. 

There were two stages in our study \cite{BBY2006} of a single equation. The first 
looked at $y = G(x,y)$ where $G(x,y)$ was a power series with non-negative coefficients. 
The second looked at more complex 
equations $y = \Theta(y)$ involving operators like Multiset, Sequence and Cycle. 
The same two stages will be 
followed in this study of generating functions defined by systems of equations. 


%
\section{Set Operations and Periodicity} \label{SectSetOps}

%
\subsection{Set operations}
The calculus of set-equations (for sets of non-negative integers) developed in this section 
uses the operations of
union ($\cup$, $\bigcup$), addition ($+$), multiplication ($\cdot$) and star ($\star$), where:

  \begin{definition}\label{set ops}
  For  $A,B\subseteq \bbN$ and $n\in \bbN$ let
\begin{eqnarray*}
     A+B & := & \{a + b : a\in A, b\in B\}\\
 n\cdot B &:=& \{nb : b\in \}\\
     n\star B&:=&
    \begin{cases}
       0&\text{for } n=0\\
       \underbrace{B+ \cdots +B}_{\text{$n$ copies of $B$}}
       & \text{ for } n > 0 

    \end{cases} \\
   A\star B &:=&   \bigcup_{a\in A} a\star B
\end{eqnarray*}
\end{definition}

The values of these operations when an argument is the empty set are: 
$\O+A = A+\O = \O$, $n\cdot \O = \O$, 
 $\O \star B = \O$, and
 $A\star \O = 0$ if $0\in A$, otherwise $A\star \O = \O$.

The obvious definition of $A\cdot B$ is not needed in this 
 study of spectra; only the special case $n\cdot B$
plays a role.
The next lemma gives the  basic identities regarding $\cup, +,\star$ 
needed for this analysis of spectra (all are easily proved).
%
\begin{lemma} \label{basic identities}  
For $A,B,C \subseteq \bbN$ and $m,n\in\bbN$
\begin{eqnarray*}
   A + (B\cup C)  & = &  (A + B) \cup (A+C) \label{equiv1}\\
   n\star(A+B)
    & = &  n\star A +  n\star B \label{equiv2}\\
%
   (A+B)\star C
    & = &  A\star C + B\star C\quad  \label{equiv3}\\
 m\star ( n\star B)
    & = &  ( m \cdot   n )\star B \label{equiv4a}\\
   (A\cup B)\star C
   & = & A\star C \cup B\star C  \label{equiv4b}\\
   A\star\big( B\cup C\big)
   &= &\bigcup_{\substack{j_1,j_2\in\bbN\\j_1+j_2\in A}}
    j_1\star B+ j_2\star C.  \label{equiv5} 
\end{eqnarray*}
\end{lemma}

%
\subsection{Periodic and eventually periodic sets}
 The following characterizations of periodic and eventually periodic sets are easily proved, if not
 well known.
 %
 %
\begin{lemma} \label{split}
Let $A\subseteq \bbN$. 
\begin{thlist}
\item
   $A$ is periodic iff there is a finite set  $A_1\subseteq \bbN$ and a positive integer $p$ 
   $($called a period for $A$$)$
    such that
   \[
     A = A_1 + p \cdot  \bbN
   \]
   iff $A$ is the union of finitely many arithmetical progressions.
\item \text{\em (Durand, Fagin, Loescher \cite{DFL1997}; Gurevich and Shelah \cite{GuSh2003})}
     $A$ is eventually periodic iff there are finite sets $A_0,A_1\subseteq \bbN$ and a positive integer $p$ 
     $($called an {\em eventual period} of $A$$)$ such that
   \[
      A = A_0 \, \cup \,  (A_1+ p \cdot \bbN)
   \]
      iff $A$ is the union of a finite set and finitely many arithmetical progressions.
\end{thlist}
\end{lemma}
 
\begin{remark}
An infinite union of arithmetical progressions 
need not be eventually periodic.
Let $U$ be the union of the arithmetical progressions 
$ a\cdot\mathbb{\bbP}$, 
where $a$ is a composite number. 
Then $U$ consists of all composite numbers. 

Given any positive integer $p$, choose a prime number $q$ that does not divide $p$. 
Then by Dirichlet's theorem the arithmetical progression $q^2 + p\cdot \mathbb{N}$ 
has an infinite number of primes, so $q^2 +p\cdot \mathbb{N}$ is not a subset of $U$.
Since $q^2\in U$, it follows that $p$ is not an eventual period for $U$ $($one
can choose $q$ arbitrarily large$)$.  Thus $U$ is not eventually periodic.
\end{remark}

\begin{lemma}\label{ground case}
Suppose $A,B,C\subseteq\bbN$ are [eventually] periodic.  
Then each of the following are [eventually] periodic:
\begin{thlist}
\item
   $A\, \cup \,  B$
 \item
   $A+ B$
 \item
   $A\star B$.
\end{thlist}
In $($c\:$)$, if $A$ is periodic and $B$ is eventually periodic, then $A\star B$ is actually periodic.
\end{lemma}

\begin{proof}
Parts (a) and (b) follow easily from Lemma \ref{split} and Lemma \ref{basic identities}.
(The eventually periodic case is discussed in \cite{GuSh2003}.)

For the eventually periodic case of (c),
choose positive integers $p,q$ so that $ p+A$ is eventually in $A$ and $q+B$ is eventually in $B$, and
use Lemma  \ref{split} to express each of $A$ and $B$ as 
the union of a finite set and finitely many arithmetical progressions, say $A = A_0 \cup (A_1+r\cdot \bbN)$
and $B = B_0 \cup (B_1 + s\cdot \bbN)$.

 Starting with
\[
   A\star B\ =\ A_0\star B \, \cup \,  (A_1+ r  \cdot  \bbN)\star B,
\]
from \eqref{equiv4b},
examine the two parts of the right side. For $A_0\star B$, using Lemma \ref{basic identities},
\begin{eqnarray*}
   A_0\star B  & = &  \bigcup_{j\in A_0} j\star \Big(B_0\, \cup \,  (B_1 + s  \cdot  \bbN )\Big)\\
   & = &   \bigcup_{j_0 + j_1 \in A_0} j_0\star B_0 +  j_1\star(B_1 + s  \cdot  \bbN)\quad
  \text{by \eqref{equiv5}} \\
   & = &  \left(\bigcup_{j_0  \in A_0} j_0 \star B_0\right) \, \cup \,  
   \left(\bigcup_{\substack{j_0 + j_1 \in A_0\\ j_1>0}} j_0 \star B_0 +  
   j_1\star B_1 + j_1 s  \cdot  \bbN\right)\quad\text{by \eqref{equiv2}} ,
\end{eqnarray*}
a union of  finitely many eventually periodic sets, hence eventually periodic by (a).

For $(A_1+ r  \cdot  \bbN)\star B$, choose $b\in B$. Then, again using Lemma \ref{basic identities},
\begin{eqnarray*}
   (A_1+ r  \cdot  \bbN)\star B
   & \supseteq &  (A_1+r + r  \cdot  \bbN)\star B\\
   & = &  r\star B + (A_1+ r  \cdot  \bbN)\star B\quad\text{by \eqref{equiv3}}\\
   &\supseteq& rb +  (A_1+ r  \cdot  \bbN)\star B .
\end{eqnarray*}
Thus $(A_1+r   \cdot  \bbN)\star B$ is actually periodic.
This shows $A\star B$ is a union of two eventually periodic sets, hence it is also
eventually periodic. 

For item (c), note that  $A$ is periodic means we can assume $A_0=\O$.  Then 
the argument for the second part above shows that $A$ is also periodic.

\end{proof}

\subsection{Periodicity parameters}

For $A\subseteq \bbN$, for $n\in \bbN$, define 
\begin{eqnarray*}
A - n &:=& \{a-n : a\in A\}\\
\gcd(0) & :=& 0\\
\fm(\O) \ :=\ \min(\O) &:=& \infty.
\end{eqnarray*}
The next definition gives some important parameters for the study of periodicity.
%
  \begin{definition} [Periodicity parameters] \label{defn per param}
  For an eventually periodic set $A\subseteq \bbN$, $A\neq \O$, let
\begin{thlist}
\item
	$\fm(A) := \min(A)$
\item
	$\fq(A) := \gcd\big(A - \fm(A)\big)$
\item
	$\fp(A)$ is the minimum of the eventual periods $p$ if $A$ is infinite; otherwise it is $0$.
\item
	$\fc(A)$ is the first element $k$ where $\fp(A)$ becomes a period for $A\cap[k,\infty)$.
\end{thlist}
\end{definition}

The following table gives the calculations of $\fm$ and $\fq$ on combinations of
                            non-empty sets
                using the  operations $\cup,+,\star$.
\begin{proposition} \label{Karen's Table}
  Let $A_1, A_2 \subseteq \mathbb{N}$ be non-empty and eventually periodic,
  with $\fm_i:= \fm(A_i), \fq_i := \fq(A_i)$, for $i=1,2$.  
  Then  
\[
\begin{array}{l @{\qquad }l @{\qquad}l}
 \text{Set} & \fm & \fq  \\
  \hline
  A_1\cup A_2 & \min (\fm_1, \fm_2) & \gcd(\fq_1, \fq_2, \fm_2-\fm_1) \\
  A_1 + A_2 & \fm_1+\fm_2 & \gcd(\fq_1, \fq_2) \\
  A_1 \star A_2 & \fm_1\fm_2& 
  \begin{cases}
  0&\text{if } A_1 = 0\\
 \gcd (\fq_2,\,  \fq_1\fm_2) & \text{if }A_1\neq 0.
 \end{cases}
\end{array}
\]
\end{proposition}

\begin{proof}
The calculations for $\fm$ are clear in each case.  

Let $a \in \bbN$, $U,V\subseteq \bbN$. Then
\begin{eqnarray*}
0\in U+V&\Rightarrow&\gcd(U+V) = \gcd(\gcd(U),\gcd(V))\\
a > 0 \in V &\Rightarrow&\gcd(a\star V) = \gcd(V).
\end{eqnarray*} 

For $A := A_1 \cup A_2$: Let $\fm := \fm(A)$, $\fq := \fq(A)$,
and suppose, without loss of generality, that $\fm_1\le\fm_2$.
Then $\fm = \fm_1$, so 
\begin{eqnarray*}
\fq&=& \gcd \Big((A_1 - \fm_1) \cup (A_2-\fm_1)\Big)
\ =\  \gcd \Big(\gcd(A_1 - \fm_1), \gcd (A_2-\fm_1)\Big)\\
&=& \gcd \Big(q_1, \gcd (A_2 - \fm_2 + \fm_2-\fm_1)\Big)
\ =\  \gcd \Big(q_1, \gcd \big(\gcd(A_2 - \fm_2), \fm_2-\fm_1)\big)\Big)\\
&=& \gcd \Big(q_1, \gcd \big(q_2, \fm_2-\fm_1)\big)\Big)
\ =\   \gcd \big(q_1, q_2, \fm_2-\fm_1)\big).
\end{eqnarray*}

For $A := A_1 + A_2$: Let $\fm := \fm(A)$, $\fq := \fq(A)$.
Then 
\begin{eqnarray*}
\fq&:=& \gcd \big(A - \fm)
\ =\   \gcd \Big((A_1 - \fm_1) + (A_2-\fm_2)\Big)\\
&=& \gcd \Big(\gcd(A_1 - \fm_1),\, \gcd (A_2-\fm_2)\Big)
\ =\   \gcd \big(q_1, q_2).
\end{eqnarray*}
  
  For $A := A_1\star A_2$: Let $\fm := \fm(A)$, $\fq := \fq(A)$.
  If $A_1=0$ then $A=0$, so $\fq=0$.
  Now suppose $A_1\neq 0$.
Then 
\begin{eqnarray*}
\fq&:=& \gcd \big(A - \fm)\\
&=& \gcd \big(A_1\star A_2 - \fm_1\fm_2 \big)\\
&=&\gcd\Big( \bigcup_{a_1\in A_1} a_1\star A_2 - \fm_1\fm_2\Big)\\
&=&\gcd\Big( \bigcup_{a_1\in A_1} a_1\star (A_2 - \fm_2) +(a_1 - \fm_1)\fm_2\Big)\\
&=&\gcd\Big\{\gcd\Big(\gcd \big(a_1\star (A_2 - \fm_2)\big),\, (a_1 - \fm_1)\fm_2\Big) : 
a_1\in A_1\Big\}\\
&=&\gcd\Big\{\gcd\Big(\gcd \big(a_1\star (A_2 - \fm_2)\big),\, (a_1 - \fm_1)\fm_2\Big) : 
a_1\in A_1, a_1\neq 0\Big\}\\
&=&\gcd\Big\{\gcd\big( \fq_2,\, (a_1 - \fm_1)\fm_2\big) : a_1\in A_1, a_1\neq 0\Big\}\\
&=&\gcd\Big(\fq_2,\, \gcd\big((A_1 - \fm_1)\fm_2\big)\Big)\\
&=&\gcd\big(\fq_2,\,  \fq_1\fm_2\big).
\end{eqnarray*}
\end{proof}

\begin{definition} 
For $A\subseteq  \bbN$ and $c\in \bbN$ let 
$
A|_{\geq c} \ :=\ A\cap [c,\infty).
$
Likewise define $A|_{> c}$, $A|_{\leq c}$, and $A|_{< c}$.
\end{definition}

The next result concerns one of the best known examples of periodic sets, 
namely in the study of the Postage Stamp Problem, also known
as the Coin Problem 
(see Example \ref{postage}).
%
\begin{lemma}\label{lin comb}
Suppose $B\subseteq \bbN$ and $B\cap \bbP\neq \O$. 
Let $A = \bbN\star B$, and let 
$\fc := \fc(A)$, $\fp := \fp(A)$, $\fq  := \fq(A)$. Then
\begin{thlist}
\item
$A$ is periodic,
\item
$\fp = \fq = \gcd(A) = \gcd(B)$, and
\item
$A= A\big|_{<\fc} \cup \Big(\fc+\fq\cdot \bbN\Big)\  \subseteq\  \fq\cdot\bbN.$
\end{thlist}
\end{lemma}

\begin{proof}
(See, e.g., Wilf \cite{Wilf1994}, $\S$3.15, for a popular proof based on analyzing the 
asymptotics for the coefficients of a generating function via partial fractions over $\bbC$---this
method originated with Sylvester.)
\end{proof}

\begin{lemma} \label{prop of p,q}
Suppose $A$ is eventually periodic. 
Letting $\fc := \fc(A)$, $\fq := \fq(A)$, $\fp := \fp(A)$, one has the following.
\begin{thlist}
\item {\em (Gurevich and Shelah \cite{GuSh2003}, Cor. 3.3)}
The set of eventual periods of $A$ is $\fp\cdot \bbP$.

\item
$\fq\,\big|\,\fp$, and $\fp = \fq$ iff $\fp \,\big|\, A - \fm$.

\item
$A$ can be expressed as the union of a finite set and a single 
arithmetical progression iff \,$\fp = \fq \cdot (A|_{\geq \fc})$.

\item
If the condition of {\em(c)} holds then
one has $A = A\big|_{< \fc} \cup (\fc +\fp\cdot \bbN)$.

\end{thlist}
\end{lemma}

\begin{proof}
The proof of (a) in \cite{GuSh2003} is  elementary, as are the following proofs for (b)--(d).
Let 
\[
\mathfrak{n} \ :=  \  \lim_{n\rightarrow \infty}\#\big( [n, n+\mathfrak{p}-1] \cap A\big).
\]
$\fn$ is well-defined since $A$ is eventually a union of arithmetical progressions by Lemma \ref{split}.  
So for $n$ of the same $\mathfrak{p}$-equvalence class the sequence stabilizes, and thus the count 
on the right side stabilizes.

(b): First we show $\fq \,\big|\, \fp$. If $A$ is finite then the proof is immediate since $\fp=0$.

Now assume $A$ is not finite.  
Take $a \in A$ large enough that by the eventual periodicity 
$a+\mathfrak{p} \in A$.  
Then $\mathfrak{q} \,|\, a - \mathfrak{m}$ and $\mathfrak{q} \,|\, a+\mathfrak{p} - \mathfrak{m}$. 
Thus $\mathfrak{q}$ divides their difference, so  
$
  \mathfrak{q}\, |\, \mathfrak{p}.
$

Clearly $\fp = \fq$ implies $\fp\,|\, (A - \fm)$ since $\fq\,|\, (A - \fm)$. Conversely, if 
$\fp\,|\, (A - \fm)$ then $\fp \,|\, \fq = \gcd(A - \fm)$; and since $\fq\,|\,\fp$ one has $\fp = \fq$. 

(c) and (d): $A$ is eventually a single arithmetical progression iff $\mathfrak{n} = 1$.  

Suppose $\mathfrak{n}=1$, then since the series defining $\mathfrak{n}$ is a 
nondecreasing series of integers on $A|_{\geq \mathfrak{c}}$ it must always be $0$ or $1$.  
Thus $A|_{\geq \mathfrak{c}}$ is precisely a single arithmetical progression giving (d).  
Further the $\mathfrak{q}$ of a single arithmetical progression is exactly the period.  
Thus $\mathfrak{q}(A|_{\geq \mathfrak{c}}) = \mathfrak{p}$

Suppose $\#( [n, n+\mathfrak{p}-1] \cap A) > 1$ for some $n\geq \mathfrak{c}$.  
Then taking the difference of two elements in this range one has
$\mathfrak{q}(A|_{\geq \mathfrak{c}}) < \mathfrak{p}$.

\end{proof}

Given $A\subseteq \bbN$ and a positive integer $n$, let
$[A]_n$ be the set of integers modulo $n$, so $[A]_n$ is a subset of
$\bbZ_n$,  the additive group of integers modulo $n$. 
Let $\overline{A}_n$ be a set of integers in $\{0,1,\ldots,n-1\}$ such
that $[\overline{A}_n]_n = [A]_n$.

%
\begin{lemma} \label{Use Zp}
Suppose $A\subseteq \bbN$ is periodic. Let $\fp := \fp(A)$.
\begin{thlist}
\item
For $m$ sufficiently large,
$$
A\big|_{\ge \fp m}\ =\ \big(\fp m+ \overline{A}_\fp\big) + \fp\cdot \bbN.
$$
\item
If $[A]_\fp$ is a subgroup of $\bbZ_\fp$ then $\fp = \fq$, 
and for $m$ sufficiently large,
$$
A\big|_{\ge \fp m}\ =\  \fp m + \fp\cdot \bbN.
$$
\end{thlist}
\end{lemma}

\begin{proof}

For (a), note that if $j\in \overline{A}_\fp$ then there is an $a \in A$ such 
that $a \equiv j \mod \fp$. Let $p$ be a period for $A$. Then 
$a +p\cdot \bbN\subseteq A$. Now $a + pn$ is eventually $\ge \fc$,
and $a+pn \equiv j \mod \fp$ (since $\fp\,\big|\,p$). Thus for $j\in \overline{A}_\fp$
there is an $a\in A$ such that $a\equiv j \mod \fp$, and $a + \fp\cdot \bbN\subseteq A$.
Writing $a = j + \fp m_j$, let $m = \max\big(m_j : j\in \overline{A}_\fp\big)$.
Then
$$
A\big|_{\ge \fp m}\ =\ \big( \fp m + \overline{A}_\fp \big) + \fp\cdot \bbN
$$

For item (b), let $[g]_\fp$ be a generator for the subgroup $[A]_\fp$, 
where $0\le g<\fp$. Then 
$$
\overline{A}_\fp \ =\ \{0,g, \ldots, (r-1)g\},
$$
where $r$ is the order of $[g]_\fp$ in $\bbZ_\fp$. 
By (a), for a sufficiently large choice of $m$ one has
$$
A\big|_{\ge \fp m}\ =\ \fp m + \{0, g, \ldots, (r-1)g\} + \fp\cdot \bbN.
$$

If $g>0$ then  $g | \fp$.  From this one has
\begin{eqnarray*}
A\big|_{\ge \fp m} 
&=& \fp m + \{0,g,\ldots, (r-1)g\} + \fp\cdot \bbN\\
&=& \fp m +  g\cdot \bbN,
\end{eqnarray*}
contradicting the fact that $\fp$ is the smallest eventual period of $A$.

Thus $g=0$, which leads to $\fp \,|\, A$ and
$$
A\big|_{\ge \fp m}\ =\ \fp m +  \fp\cdot \bbN.
$$
From $\fp \,|\, A$ one has $\fp \,|\, A-\fm$, and thus $\fp = \fq$, by Lemma \ref{prop of p,q} (b).

\end{proof}

The next lemma augments the results of Lemma \ref{prop of p,q}\,(c),(d), giving a simple
condition that is sufficient to guarantee that $A$ is a periodic set 
involving a single arithmetical progression. This is used in the study of non-linear
systems defining generating functions.

\begin{lemma} \label{dbl lem}
Suppose $A\subseteq\bbN$ with $A\cap\bbP\neq \O$,  and suppose there
are integers $r\ge 0$ and $s\ge 2$ such that
\begin{equation*}\label{dbl cond}
  A \  \supseteq \  r + s\star A .
\end{equation*}
 Let $\fc := \fc(A)$, $\fm := \fm(A)$, $\fp := \fp(A)$ and $\fq := \fq(A)$.
Then $A$ is a periodic set, $\fp =\fq$, and
  $$
 A \ =\  A\big|_{< \fc} \cup (\fc +\fp\cdot \bbN).
 $$
  \end{lemma}
  
  \begin{proof}
Choose $t\in (s-1)\star A$. 
Then
  	$A \  \supseteq\  r + t + A$, 
so $A$ is periodic.

 Next let $B := A - \fm$, a subset of $\bbN$ with 0 in it (since $\fm\in A$).
 Furthermore $B$ is periodic and $\fp(B) = \fp$, $\fq(B) = \fq$.
 Letting $b = r + (s-1)\fm $,
 \begin{align*}
B
 &=\ A - \fm \ \supseteq\  r + s\star A - \fm,  \nonumber
 \intertext{so}
B &\supseteq\ b + s\star B. \nonumber
\intertext{Since $0 \in B$ and $s \ge 2$,}
B&\supseteq\ b + B\quad\text{and}\quad B\ \supseteq\ b + B + B. \nonumber
\intertext{From this one easily derives}
B\ \supseteq\ b\:\fp + B +  B,
\intertext{so reducing modulo $\fp$,}
[B]_\fp\ \supseteq\  [B]_\fp +  [B]_\fp.
   \end{align*}
 This means $[B]_\fp$ is a subgroup of $\bbZ_\fp$, so, by Lemma \ref{Use Zp}\,(b),
 $\fp = \fq$, and
 for $m$ sufficiently large,
 $$
 B\big|_{\ge \fp m}\ =\  \fp m +  \fp\cdot \bbN,
 $$
 which gives
  $$
 A\big|_{\ge \fm + \fp m}\ =\  \fm + \fp m +  \fp\cdot \bbN.
 $$
 But then, by Lemma \ref{prop of p,q} (d),
 $A = A\big|_{< \fc} \cup (\fc +\fp\cdot \bbN)$.

  \end{proof}

\section{Systems of Set-Equations}

We will consider systems of set-equations of the form
\begin{eqnarray*}
Y_1&=&\gam_1(Y_1,\ldots,Y_k)\\
&\vdots&\\
Y_k&=&\gam_k(Y_1,\ldots,Y_k),
\end{eqnarray*}
written compactly as $\bY = \bgam(\bY)$, with the $\gam_i(\bY)$ 
having a particular form, namely
\begin{equation} \label{Gamma form}
 \gam_i(\bY) \ =\ \bigcup_{\bu\in\bbN^k} \gam_{i,\bu} + u_1\star Y_1 +\cdots + u_k\star Y_k,
\end{equation}
where the $\gam_{i,\bu}$ are subsets of $\bbN$.
The system of equations \eqref{Gamma form} is simply expressed by
\begin{equation} \label{bGamma form}
 \bgam(\bY) \ =\ \bigcup_{\bu\in\bbN^k} \bgam_{\bu} + \bu \star \bY ,
\end{equation}
where 
$\bu\star\bY := u_1\star Y_1 +\cdots + u_k\star Y_k$.

\subsection{$\Sdom$ and $\Sdom_0$}
\begin{definition}\label{SetDom}
Let $\Sdom$ be the set of $\bgam(\bY)$ of the form \eqref{bGamma form}, and let
$\Sdom_0$ be the set of $\bgam(\bY)\in\Sdom$ which map $\Su(\bbP)^k$ into itself.
A system $\bY = \bgam(\bY)$ of set-equations  is {\em basic} 
if $\bgam(\bY) \in \Sdom_0$.
\end{definition}

%
\begin{lemma} \label{Gamma prop}
Suppose $\bgam(\bY)\in\Sdom$. Then
\begin{thlist}
\item
for $\bA \in \Su(\bbN)^k$ one has
$$
 \gam_i(\bA) \ =\ \bigcup_{\bu\in\bbN^k}\Big( \gam_{i,\bu} + \sum_{\{j : u_j>0\}} u_j\star A_j \Big) ,
 $$
 where the summation term is omitted in the case that all $u_j = 0$.
 \item
$\bA \in \Su(\bbP)^k$ and $0\in \bu\star \bA$ imply $\bu = \mathbf{0}$.
\item
$\bgam(\bY)\in\Sdom_0$ iff $\,\bgam_{\mathbf 0} \in \Su(\bbP)^k$.
 \end{thlist}
\end{lemma}

\begin{proof}
(a) follows from the fact that $0\star A_j = 0$, by Definition \ref{set ops}.

Given $\bA \in \Su(\bbP)^k$,
 (b) follows from
 \begin{eqnarray*}
0\in \bu\star\bA
&\Leftrightarrow& 0\in u_i\star A_i \quad \text{for }1\le i\le k\\
&\Leftrightarrow& u_i = 0 \quad\text{for } 1\le i \le k,
\end{eqnarray*}
 the last assertion holding because $0\notin A_i$ for any $i$, and Definition \ref{set ops}.

For (c),
let $\bA\in\Su(\bbP)^k$. Then 
\begin{eqnarray*}
\bgam(\bA) \subseteq \Su(\bbP)^k
&\Leftrightarrow &
0\notin \Gamma_i(\bA)\quad\text{for } 1\le i \le k\\
&\Leftrightarrow &
0 \notin  \gam_{i,\bu} + \bu\star\bA
\quad\text{for } 1\le i \le k, \ \bu\in\bbN^k\\
&\Leftrightarrow &
0 \notin  \gam_{i,\bu} \cap \bu\star\bA
\quad\text{for } 1\le i \le k, \ \bu\in\bbN^k\\
&\Leftrightarrow &
\Big( 0 \in  \bu\star\bA \Rightarrow  0 \notin \gam_{i,\bu}\Big)
\quad\text{for } 1\le i \le k, \ \bu\in\bbN^k\\
&\Leftrightarrow &
0 \notin \gam_{i,\mathbf{0}}
\quad\text{for } 1\le i \le k,
\end{eqnarray*}
the last line by item (b).
\end{proof}

Define a partial ordering $\unlhd$ on $\Sdom$ by
$$
\bgam(Y)\ \unlhd\ \bDelta(Y)\ \text{ iff }\  
\Gamma_{i,\bu} \subseteq \Delta_{i,\bu}\quad \text{for all } i, \bu.
$$
$\bgam^{(n)}(\bY)$ denotes the $n$-fold composition of $\bgam(\bY)$ with itself, and 
$\gam_{\ i}^{(n)}(\bY)$ is the $i$th component of this composition. 
Let $\bgam^{(\infty)}(\bY) := \bigcup_{n\ge 0}\gam^{(n)}(\bY) $.
For $\bA,\bB\in\Su(\bbN)^k$ let,
\begin{itemize}
\item
$\min \bA := (\min A_1,\ldots,\min A_k)$
\item
$\bA \le \bB$ expresses $A_i \subseteq B_i$ for $1\le i\le k$
\item
$\cN(\bA) := \{i :A_i = \O\}$.
\end{itemize}

\begin{lemma} \label{catch-all}
Given $\bgam(\bY)\in\Sdom$, and $\bA,\bB\in \Su(\bbN)^k$, the 
following hold:
\begin{thlist}
\item
$
\bA  \le \bB \ \Rightarrow\ \bgam\big(\bA \big)\le \bgam\big(\bB \big)
$

\item
$
\cN(\bA) = \cN(\bB)\ \Rightarrow\ 
\cN\big(\bgam(\bA)\big) = \cN\big(\bgam(\bA)\big)
$
\item
$
\bA \le \bB \  \Rightarrow  \ \cN\big(\bgam(\bA)\big) \supseteq \cN\big(\bgam(\bB)\big)
$

\item
$\cN\big(\bgam^{(k)}(\bempty)\big) = \cN\big(\bgam^{(k+n)}(\bempty)\big)$ for $n\ge 0$.
\end{thlist}
\end{lemma}

\begin{proof}
Item (a) follows from the montonicity of the set operations $\bigcup, + , \star$ used in the 
definition of the $\bgam(\bY)$ in $\Sdom$.

Next observe that
\begin{equation}\label{NGamma}
\cN\big(\bgam(\bA)\big)\ =\ \bigg\{ i : \big(\forall \bu \in \bbN^k\big) 
\Big(\gam_{i,\bu} = \O\ \text{or}\ \big(\exists j \big)\big(u_j>0 \text{ and } A_j = \O\big)\Big)\bigg\},
\end{equation}
since from \eqref{bGamma form}  one has
 $i\in \cN\big(\bgam(\bA)\big)$ iff for every $\bu\in\bbN^k$ one has
$\gam_{i,\bu} + \bu \star \bA = \O$, and this holds iff for every
$\bu\in\bbN^k$ one has either
$\gam_{i,\bu}=\O$, or for some $j$, $u_j \star A_j = \O$.
Note that $u_j \star A_j = \O$ holds iff $u_j > 0$ and $A_j = \O$.

Item (b) is immediate from \eqref{NGamma}.

To prove (c), note that $B_j = \O\Rightarrow A_j=\O$, and then use \eqref{NGamma}.

To prove (d), note that from $\bempty \le \bgam(\bempty)$ and (a) one has an increasing sequence
$$
\bempty \le \bgam(\bempty) \le \bgam^{(2)}(\bempty)\le \cdots.
$$
Then (c) gives the decreasing sequence
$$
\{1,\ldots,k\}\  =\  \cN(\bempty)\  \supseteq\  \cN\big(\bgam(\bempty)\big)\ \supseteq\ \cN\big(\bgam^{(2)}(\bempty)\big) \ \supseteq\ \cdots.
$$
From (b) one sees that once two consecutive members of this sequence are equal, 
then all members further along in the sequence are equal to them. 
This shows the sequence must stabilize by the term $\cN(\bgam^{(k)}(\bempty))$.
\end{proof}

For the next lemma, recall that $\min(\O) := +\infty$.

\begin{lemma} \label{min expr}
Suppose $\bgam \in \Sdom_0$, and suppose $\bA\subseteq \Su(\bbP)^k$ 
with $\bA \le \bgam(\bA)$.
Then
$$
\min{\bgam^{(\infty)}(\bA)}\ =\ \min{\bgam^{(k)}(\bA)}. 
$$
In particular,
$
\min{\bgam^{(\infty)}(\bempty)}\ =\ \min{\bgam^{(k)}(\bempty)}. 
$
\end{lemma}
\begin{proof}
From
$$
\bgam(\bY)\ :=\  \bigcup_{\bu\in\bbN^k} \bgam_{\bu} + \bu \star \bY
$$
one has, by Lemma \ref{Gamma prop}\,(a), for $1\le i \le k$,
\begin{equation*}   \label{Gam exp}
 \gam_i^{(n+1)}(\bA) \ =\  \bigcup_{\bu\in\bbN^k} \Big( \gam_{i,\bu} + \sum_{ \{ j : u_j > 0 \} } 
 u_j\star \gam_{\ j}^{(n)}(\bA) \Big). 
 \end{equation*}
Let
\begin{eqnarray*}
\bb_\bu  &:=& \min \bgam_\bu(\bA)\\
 \bb^{(n)} & :=& \min \bgam^{(n)}(\bA),
\end{eqnarray*}
that is, for $1\le i \le k$,
\begin{eqnarray*}
b_{i,\bu} &=& \min \gam_{i,\bu}(\bA)\\
 b_i^{(n)} &:=& \min \gam_{\ i}^{(n)}(\bA).
\end{eqnarray*}
Then for $n\ge 0$,
\begin{equation} \label{bdec}
\bb^{(n+1)}\le \bb^{(n)},
\end{equation} 
since 
$\bA \le \bgam(\bA)$ implies
$\bgam^{(n)}(\bA) \le \bgam^{(n+1)}(\bA)$, by repeated application of Lemma \ref{catch-all}\,(a).

From the above,
\begin{align}
b_i^{(n+1)}
&:=\  \min \gam_{\ i}^{(n+1)}(\bA)\nonumber \\
&=\  \min  \bigcup_{\bu\in\bbN^k} \Big( \gam_{i,\bu} + \sum_{ \{ j : u_j > 0 \} } u_j\star \gam_{\ j}^{(n)}(\bA) \Big) \quad\text{by }\eqref{Gamma form},\nonumber 
\intertext{so}
b_i^{(n+1)} &=\  \min\Big\{ b_{i,\bu} + \sum_{ \{ j : u_j > 0 \} } u_j b_j^{(n)} : \bu \in \bbN^k\Big\} . \label{bin expr}
\end{align}

For $n\ge 1$ let 
\begin{equation}
I_n\  :=\   \Big\{ j : b_j^{(n)} < b_j^{(n-1)} \Big\}.
\end{equation}

CLAIM:
\begin{equation*}\label{step down}
    (\forall n\ge 1)(\forall i\in I_{n+1})
     (\exists r\in I_{n})\Big( b_i^{(n+1)} \ge b_r^{(n)}\Big).
\end{equation*}
\begin{proof}[Proof of Claim]
Suppose $n\ge1$ and $i\in I_{n+1}$, that is, 
$$
b_i^{(n+1)} < b_i^{(n)}.
$$
From \eqref{bin expr}, let $\bu\in \bbN^k$ be such that
\begin{equation}\label{b loc}
b_i^{(n+1)} \ =\ b_{i,\bu} + \sum_{ \{ j : u_j > 0 \} } u_j b_j^{(n)} .
\end{equation}
Let $r \in \{1,\ldots,k\}$ be such that $u_r > 0$ and $r \in I_n$---such an $r$
must exist, 
for otherwise $u_j>0$ would imply $j\notin I_n$, that is, $b_j^{(n)} = b_j^{(n-1)}$.
 Then from \eqref{b loc}, and from \eqref{bin expr} with $n-1$ substituted for $n$,
 \begin{eqnarray*}
 b_i^{(n+1)}  
 &=& b_{i,\bu} + \sum_{\{j : u_j>0\}} u_j b_j^{(n-1)}
 \ \ge\ b_i^{(n)},
 \end{eqnarray*}
contradicting the assumption that $i\in I_n$, that is,
$b_i^{(n+1)} < b_i^{(n)}$.

For this choice of $\bu$ and $r$, \eqref{b loc} implies
\begin{equation}
b_i^{(n+1)}  \ \ge \ b_r^{(n)},
\end{equation}
 establishing the Claim.
\end{proof}

Now suppose $I_{n}\neq \O$ for some $n\ge k+1$. 
Then, by the Claim,  one can choose a sequence $i_n,\ldots,i_{n-k}$ of indices 
from $\{1,\ldots,k\}$
such that 
\begin{eqnarray}\label{dec seq}
b_{i_n}^{(n)} \ge b_{i_{n-1}}^{(n-1)} \ge \cdots \ge b_{i_{n-k}}^{(n-k)}
\end{eqnarray}
and $i_j \in I_j$ for $n-k \le j \le n$.
By the pigeonhole principle there are two  $j$ such that the
indices $i_j$ are the same, say $r = i_p = i_q$, where $n-k\le p < q \le n$. 
Then
 $b_r^{(q)} \ge b_r^{(p)}$ by \eqref{dec seq}.  
 But from $r\in I_q$ and \eqref{bdec} one has $b_r^{(q)} < b_r^{(q-1)}\le \cdots \le b_r^{(p)}$,
 giving a contradiction.
 Thus $I_{n}=\O$ for $n>k$, completing the proof of the lemma.
   \end{proof}

\subsection{The Minimum Solution of $\bY = \bgam(\bY)$}

\begin{proposition}\label{min prop}
For $\,\bgam(\bY)\in \Sdom$, the system of set-equations $ \bY = \bgam(\bY)$ 
has a minimum solution $\bS$, and it is given by
\begin{equation*}\label{smallest sol}
\bS \ =\ \bgam^{(\infty)}(\bempty)\ :=\ \bigcup_{n\ge 0} \bgam^{(n)}(\bempty).
\end{equation*}
If $\,\bgam(\bY)\in \Sdom_0$ then, for $1\le i \le k$, one has $S_i = \O\ $ iff  $\ \gam_{\ i}^{(k)}(\bempty) = \O$.
\end{proposition}

\begin{proof}
The sequence of sets $\bgam^{(n)}(\bempty)$ is non-decreasing 
by Lemma \ref{catch-all}\,(a) since 
$\bempty \subseteq \bgam(\bempty)$.
Suppose $a \in \gam_{\ i}^{(\infty)}(\bempty)$. Then, for some $n\ge 1$,
\begin{equation}
a \in \gam_{\ i}^{(n)}(\bempty) \ =\  \gam_i\big(\bgam^{(n-1)}(\bempty)\big) 
\ \subseteq\  \gam_i\big(\bgam^{(\infty)}(\bempty)\big).
\end{equation}
This implies $\bgam^{(\infty)}(\bempty) \le \bgam\big(\bgam^{(\infty)}(\bempty) \big)$.

Conversely, suppose $a \in \gam_i\big(\bgam^{(\infty)}(\bempty) \big)$.
Then for some $\bu\in \bbN^k$, 
\begin{equation}
a \in \gam_{i,\bu} +   \bu\star\bgam^{(\infty)}(\bempty),
\end{equation}
which in turn implies  for some $\bu\in \bbN^k$ and $n\ge 1$,
\begin{equation}
a \in \gam_{i,\bu} +   \bu\star\bgam^{(n)}(\bempty) \  
\subseteq\   \bgam^{(n+1)}(\bempty)\   \subseteq\   \bgam^{(\infty)}(\bempty).
\end{equation}
Thus
$\bgam^{(\infty)}(\bempty) = \bgam\big(\bgam^{(\infty)}(\bempty)\big)$, so 
$\bgam^{(\infty)}(\bempty)$ is indeed a solution to $ \bY = \bgam(\bY)$.

Now, given any solution $\bT$, from $\bempty \le \bT$ and Lemma \ref{catch-all}\,(a) it follows
that for $n\ge 0$, $\bgam^{(n)}(\bempty) \subseteq \bgam^{(n)}(\bT) = \bT$, and
thus $\bgam^{(\infty)}(\bempty) \subseteq \bT$, showing that $\bgam^{(\infty)}(\bempty)$
is the smallest solution to $ \bY = \bgam(\bY)$.

The test for $S_i = \O$ is immediate from Lemma \ref{min expr}.

\end{proof}

%
\subsection{The Dependency Digraph for $\bY = \bgam(\bY)$}

In the study of systems $\bY = \bgam(\bY)$ with $\bgam(\bY)\in\Sdom$, 
it is important to know when 
$Y_i$ depends on $Y_j$.
This information is succinctly collected in the dependency digraph of the system.

\begin{definition}
 The {\em dependency digraph} $D$ of a system $\bY = \bgam(\bY)$
$($with $k$ equations\,$)$ has vertices $1,\ldots,k$ 
 and directed edges given
 by $i \rightarrow j$ iff  there is a $\bu\in\bbN^k$ such that $\gam_{i,\bu} \neq\O$
 and $u_j>0$. 
 
 The {\em dependency matrix} $M$ of the system is the matrix of the digraph $D$.
 
  If $i\rightarrow j \in D$
 then we say ``\,$i$ depends on $j$'', as well as ``\,$Y_i$ depends on $Y_j$''.
The transitive closure of $\rightarrow$ is $\rightarrow^+ $; the notation $i\rightarrow^+ j$ 
is read: ``\,$i$ eventually depends on $j$''. It asserts that there is a directed
 path in $D$ from $i$ to $j$. 
In this case one also says ``\,$Y_i$ eventually depends on $Y_j$''.
The reflexive and transitive closure of $\rightarrow$  is  $\raStar $.

For each vertex $i$ let $[i]$ denote the $($possibly empty$)$ {\em strong component} of $i$
in the dependency digraph, that is, 
$$
[i] \  :=  \  \big\{ j : i\rightarrow^+   j\rightarrow^+i \big\}.
$$
 \end{definition}

For a given system $\bY = \bgam(\bY)$, the following are 
easily seen to be equivalent:
\begin{thlist}
\item
$i\rightarrow^+ j$

\item
there is an $n \in \{1,\ldots,k\}$ such that $(M^n)_{i,j}=1$.

\item
the $(i,j)$ entry of $M +\cdots+M^n$ is not $0$.
\end{thlist}

%
\subsection{The Main Theorem on Set Equations}

Recall that $\bu\star \bY$ means 
$u_1\star Y_1 + \cdots + u_k\star Y_k$; 
 and
$\bgam_\bu + \bu\star \bY$ 
is the $k$-tuple obtained by adding $\bu\star\bY$ to each
component of $\bgam_\bu$.

It is well-known that $\big(\Su(\bbN)^k,d\big)$ is a complete metric space, where
$$
d(\bA,\bB)\  :=\  
\begin{cases}
2^{-\min\,\bigcup_{i=1}^k\big(A_i \triangle B_i\big)} &\text{if } \bA\neq \bB\\
0&\text{if }\bA = \bB.
\end{cases}
$$
 In this space $\displaystyle \lim_{n\rightarrow \infty} d(A_n,B_n) =0$ iff  for 
 every $m$ there is an $N$ such that
$A_n\big|_{\le m} = B_n\big|_{\le m}$ for $n\ge N$.

When the minimum solution $\bS$ of a basic system $\bY=\bgam(\bY)$
 is meant to give spectra $S_i$ of generating functions, 
then 0 is excluded from the $S_i$, so one has the condition
$0 \notin \gam_{i,\bzero}, 1\le i\le k.$
Also
one can assume that trivial equations $Y_i = Y_j$ have, 
after suitable substitutions
into the other equations,  been set aside. Thus one
can assume there are no terms $\gam_{i,\bu} +   \bu\star\bY$ 
which are simply a variable $Y_j$.
Both restrictions on $\bGamma(\bY)$ are captured in the definition of {\em elementary\,} systems
   of set-equations.
   
   \begin{definition} A basic system $\bY = \bgam(\bY)$ of set-equations
    is an {\em elementary}
   system if it satisfies
$$
  0\in \gam_{i,\bu}
   \ \Rightarrow\  \sum_{j=1}^k u_j  \ge 2\quad
   \text{ for }1\le i\le k, \bu\in\Su(\bbN)^k.
$$
If it also satisfies $\cN\big(\bGamma^{(k)}(\bempty)\big) = \O$, that is,
no coordinate of $\bGamma^{(k)}(\bempty)$ is the empty set, then 
one has a {\em reduced} elementary system.
\end{definition}

If $\bY = \bgam(\bY)$ is a non-reduced elementary system, then a simple
process of {\em reduction} allows one to eliminate the $Y_i$ for which 
$i\in\cN\big(\bGamma^{(k)}(\bempty)$, namely by
substituting $\O$ for all occurrences of such $Y_i$ in $\bgam(\bY)$, and
removing the equations with such $Y_i$ on the left side. The resulting
system will be reduced elementary.

\begin{theorem}\label{ThmSetEqn}
Let
 $\bY = \bgam(\bY)$ be an elementary system of $k$ set-equations.
Then the following hold:
\begin{thlist}

\item
There is a unique solution $\bT \in \Su(\bbP)^k$, and it is given by
$$
\bT \ =\ \bgam^{(\infty)}(\bA)\ :=\ 
\lim_{n\rightarrow \infty}\bgam^{(n)}(\bA), \ \text{for any } \bA\in\Su(\bbP)^k.
$$

\item
$T_i= \O\ $ iff  $\ \gam_{\ i}^{(k)}(\bempty) = \O$, that is, $i\in\cN\big(\bGamma^{(k)}(\bempty))$.

\end{thlist}
For the remaining items, we assume the system is reduced.
\begin{thlist}
\item [c]
 $[i] \neq\O$ implies $T_i$ is periodic. If also there is a $j\in[i]$ such that
 for some $\bu\in\bbN^k$ one has
 $\gam_{j,\bu}\neq \O$ and $\sum\{u_\ell : \ell \in [i]\} \ge 2$, then $T_i$ is the union of
 a finite set with a single arithmetical progression.

\item[d]
 Suppose $[i] = \O$ and the $i$th equation can be written in the form
 $$
 Y_i\ :=\ P_i \  +\ \bigcup_{Q \in \fQ_i} \sum_{j=1}^k Q_j\star Y_j,
 $$
 with $P_i$ [eventually] periodic, and with  $\fQ_i$ a finite set of 
 $k$-tuples $Q = (Q_1,\ldots,Q_k)$ of 
 [eventually] periodic subsets $Q_j$ of $\bbN$, and 
 for $i \rightarrow j$ one has $T_j$ being [eventually] periodic.
 Then
 $T_i$ is [eventually] periodic.
%
\item[e]
The periodicity parameters $\bfm,\bfq$ of the solution $\bT$ can be found from 
$\bgam^{(k)}(\bempty) $ and
the 
$\bGamma_\bu$ via the formulas:
\begin{eqnarray}
 	\fm_i :=  \fm_i(T_i) 
		&=& \min\Big(\gam_{\ i}^{(k)}(\bempty)\Big)  \label{m formula}\\
	\fq_i := \fq(T_i) 
		&=&\gcd \bigg(\bigcup_{i\raStar  j}
 		\bigcup_{\bu\in\bbN^k} \Big(\gam_{j,\bu} + 
		\bu\star\bfm  -\fm_j \Big) \bigg)  .              \label{q formula}
 \end{eqnarray}

 \item[f]
$\fq_i\, |\, \fq_j$ whenever $i\rightarrow j$.
\end{thlist}
\end{theorem}

 \begin{proof}
 The mapping 
 $\bgam : \Su(\bbN)^k \rightarrow \Su(\bbN)^k$ is a contraction
map on the complete metric space $\big(\Su(\bbN)^k,d\big)$, proving (a).
Item (b) follows from Proposition \ref{min prop}.
 
 Now we are assuming that the system is reduced. 
 For (c), first note that given $i$ and  $\bu$ such that $\gam_{i,\bu}\neq \O$,  
there is a $q\ge 0$ (any $q\in\gam_{i,\bu}$) such that
\begin{equation*}\label{sec}
 T_i\ \supseteq\ 
 q+ \sum_{\substack{1\le j\le k\\u_j\neq 0}}u_j\star T_j .
\end{equation*}
From this, 
 $i\rightarrow j$ implies $T_i\supseteq p + T_j$ for some positive $p$, hence
\begin{equation*}\label{secp}
i\rightarrow^+ j\quad\text{implies}\quad T_i \supseteq p+ T_j\quad\text{for some positive } p.
\end{equation*}
 Now suppose $[i]\neq\O$.  
 Then  $i\rightarrow^+i$, so $T_i \supseteq p + T_i$ for some positive $p$, 
 that is, $T_i$ is periodic.

For the second part of (c), one can assume  that  $\rightarrow $ equals $\rightarrow^+$ 
(by using $\bgam \cup \cdots \cup \bgam^{(k)}$ in place of $\bgam$). 
From $i\rightarrow j$ follows $T_i \supseteq p_1 + T_j$ for some positive $p_1$. 
The hypothesis of (c) gives
$T_j \supseteq p_2 + T_a + T_b$ for some $a,b$ (possibly equal) and some $p_2\ge 0$. 
Finally $a\rightarrow i$ and $b\rightarrow i$
show that $T_a \supseteq p_3 + T_i$ and $T_b \supseteq p_4 + T_i$ for positive $p_3,p_4$. 
With $p=p_1+p_2+p_3+p_4$ one has 
$T_i \supseteq p + 2\star T_i$.
Then Lemma \ref{dbl lem} gives the desired conclusion.
 For (d), just apply Lemma \ref{ground case}.

Now to prove (e) and (f).
 The expression \eqref{m formula} for $\fm_i$ is given in Lemma \ref{min expr},
  so it remains to derive
 the formula \eqref{q formula} for $\fq_i$.
$\bT$ is the unique solution to the system, so
\begin{eqnarray*}\label{Eq2}
	  \bT&=& \bigcup_{\bu\in\bbN^k}  \Big( \bgam_\bu + \bu\star \bT \Big). 
\end{eqnarray*}
 Letting $\bS = \bT - \bfm$, one has $0$ in each $S_i$ and
\begin{eqnarray*}\label{Eq3}
	 \bS &=& \bigcup_{\bu\in\bbN^k}   
	 \Big( \bgam_\bu  + \bu\star \bfm -\bfm + \bu\star\bS \Big),
\end{eqnarray*}
or in terms of the individual components $S_j$ one has,  
$$
 	S_j \   =  \  \bigcup_{\bu\in\bbN^k}   \bigg( \gam_{j,\bu} + 
	\Big(\sum_{\ell=1}^k u_\ell \fm_\ell \Big) - \fm_j  +
 	\sum_{\ell=1}^k u_\ell \star S_\ell \bigg).
$$
For $1\le j\le k$ let
$$
R_{j,\bu} \ :=  \  \bigcup_{\bu\in\bbN^k} 
\bigg(\gam_{j,\bu} + \Big( \sum_{\ell=1}^k u_\ell\fm_\ell \Big) - \fm_j  \bigg)\,,
$$
so 
\begin{eqnarray}\label{Eq3a}
S_j\   =  \  \bigcup_{\bu\in\bbN^k} \Big( R_{j,\bu} + \sum_{\ell=1}^k u_\ell \star S_\ell\Big).
\end{eqnarray}
Since $0\in u_\ell\star S_\ell$ for all $\ell$, one has for $1\le j \le k$ and $\bu\in\bbN^k$, 
\begin{eqnarray}\label{Eq3A}
 	S_j &\supseteq& R_{j,\bu}.
\end{eqnarray}
By definition, $q_i = \gcd(S_i)$, so \eqref{Eq3A} implies
\begin{equation}\label{Eq3b}
	\fq_i \,\Big|\,  R_{i,\bu}.
\end{equation}
For $i\rightarrow j$ there is a $\bu\in\bbN^k$ such that
$\gam_{i,\bu}\neq\O$. 
Then 
\eqref{Eq3a} and  \eqref{Eq3b} imply that 
\begin{equation}\label{Eq3c}
	\fq_i\,\big | \,S_j\  \text{whenever } i\rightarrow j ,
\end{equation}
since $\fq_i \big | \, S_i$, and since whenever $\gam_{i,\bu}\neq \O$ one has 
$
S_i   \supseteq R_{i,\bu} + S_j.
$
This proves item (f) of the theorem. 
From \eqref{Eq3a} and \eqref{Eq3c} 
\begin{equation}\label{Eq3d}
	i\rightarrow^+  j \  \Rightarrow\ 
	\fq_i \,\Big|\,  R_{j,\bu}.
\end{equation}
From \eqref{Eq3b} and \eqref{Eq3d} 
\begin{equation}\label{Eq3e}
	\fq_i \,\Big|\, \fq_i^\star \ :=  \ 
		\gcd\bigg(\bigcup_{i\raStar  j  } R_{j,\bu}\bigg).
\end{equation}
To show $\fq_i^\star \,\big|\, \fq_i$, 
from
$
    \bT   =   \bgam^{(\infty)}(\bempty),
$
 one has 
 $
    \bS  =   \bOmega^{(\infty)}(\bempty),
$
 where
 \[
   \bOmega :\  \bA \mapsto\   \bigcup_{\bu\in\bbN^k} 
   \bgam_\bu -\bfm + \bu\star \bfm  + \bu\star \bA\   =  \  \bR_\bu + \bu\star \bA.
 \]
One proves, by induction on $n$, that  
 \begin{equation*}\label{div ind}
 i\raStar  j
  \  \Rightarrow\   
 \fq_i^\star\,\Big|\,\Omega_{\ j}^{(n)}(\bempty).
 \end{equation*}
  
 {\bf Ground Case:} (n=1)\\
   $
  \Omega_j(\bempty)   =  R_{j,\bu}
   $
 so $\fq_i^\star \,\Big|\,   \Omega_j(\bempty)$ if
 $i\raStar  j$, by the definition of $\fq_i^\star$ in \eqref{Eq3e}.
 
 {\bf Induction Step:}\\
 Assume that $\fq_i^\star \,\Big|\, \Omega_{\ j}^{(n)}(\bempty)$ if
$i\raStar  j$. One has
 $$
\Omega_{\ j}^{(n+1)}(\bempty) \  =  \ 
 \bigcup_{\bu\in\bbN^k} \Big(R_{j,\bu} + 
   \sum_{\ell=1}^k u_\ell \star \Omega_{\ \ell}^{(n)}(\bempty)\Big).
 $$
Suppose that  $i\raStar  j$. 
Then $q_i^\star\,\big|\,R_{j,\bu}$ (by  the definition of $q_i^\star$ in \eqref{Eq3e}).
For $\bu\in\bbN^k$, clearly 
\begin{equation}\label{q0}
   R_{j,\bu}\   =  \  \O\  \Rightarrow\  \fq_i^\star \,\Big|\,\Big(R_{j,\bu} + 
   \sum_{\ell=1}^k u_\ell \star \Omega_{\ \ell}^{(n)}(\bempty)\Big)\   =  \  \O. 
\end{equation}
If $\bu\in\bbN^k$ is such that $R_{j,\bu}\neq \O$, let $u_m>0$. 
Then $j\rightarrow m $, and since $i\raStar  j$, 
one has $i\raStar   m$.
By the induction hypothesis this implies $\fq_i^\star \,\Big|\, \Omega_{\ m}^{(n)}(\bempty)$. 
Consequently $\displaystyle \fq_i^\star \,\Big|\, \sum_{\ell=1}^k 
u_\ell \star \Omega_{\ \ell}^{(n)}(\bempty)$,
and one knows $q_i^\star\,\big|\,R_{j,\bu}$.
Thus
\begin{equation}\label{q1}
   R_{j,\bu}\  \neq\  \O\  \Rightarrow\   \fq_i^\star \, \Big|\,\Big(R_{j,\bu} + 
   \sum_{\ell=1}^k u_\ell \star \Omega_{\ \ell}^{(n)}(\bempty)\Big) .
\end{equation}
Items \eqref{q0} and \eqref{q1} show that for $\bu\in\bbN^k$,
$$
 i\raStar  j\  \Rightarrow\  \fq_i^\star \,\Big|\,\Big(R_{j,\bu} + 
   \sum_{\ell=1}^k u_\ell \star \Omega_{\ \ell}^{(n)}(\bempty)\Big), 
$$
so
$$
 i\raStar  j\  \Rightarrow\   \fq_i^\star \, \Big|\, 
\Omega_{\ \ell}^{(n+1)}(\bempty)\   =  \ 
  \bigcup_{\bu\in\bbN^k} \Big(R_{j,\bu} + 
   \sum_{\ell=1}^k u_\ell \star \Omega_{\ \ell}^{(n)}(\bempty)\Big)
$$
finishing the induction proof.
Thus 
$$
    i\raStar  j\  \Rightarrow\   \fq_i^\star \, \Big|\, 
 \Omega_{\ j}^{(\infty)}(\bempty) \   =  \ S_j\,.
 $$
In particular, 
$
   \fq_i^\star\,\big|\, S_i
$
so
$
    \fq_i^\star\,\big|\, \fq_i    =    \gcd(S_i),
$
completing the proof.
\end{proof}
 
%
\section{Elementary Power Series Systems}
%
%
\subsection{General Background for Power Series Systems}
Recall that
$\bbR$ is the set of reals, $\bbN$ the set of non-negative integers, and 
$\bbP$ the set of positive integers. The following table gives the notations 
needed for this section:
%
{\renewcommand{\arraystretch}{1.3}
$$
\begin{array}{|l c l|}
\hline
\bz &=& z_1,\ldots,z_m\\
\bbF&=&\text{a field}\\
\bbF[[\bz]]&=&\text{set of power series }A(\bz) = \sum_\bu a_\bu \bz^\bu\text{  over $\bbF$}\\
\bbF[[\bz]]^k&=&\{(A_1(\bz),\ldots,A_k(\bz)) : A_i(\bz) \in \bbF[[\bz]] \}\\
\bbF[[\bz]]_0&=&\{A(\bz) \in \bbF[[\bz]] : A(\bzero) = 0\}\\
\left[ x^{\le m}\right] A(x)&=&a(0) +a(1)x + \cdots + a(m)x^m\\
J_\bG(x,\by)&=&\text{the Jacobian matrix of }\bG(x,\by) \text{ with respect to $\by$}\\
\Spec(T(x))&=& \{n\ge 0 : t(n)\neq 0\}, \quad\text{for }T(x)\in\bbF[[x]]\\
\Spec(\bT(x))&=&\big(\Spec(T_1(x)),\ldots,\Spec(T_m(x))\big),\quad\text{for } \bT(x)\in\bbF[[x]]^m\\
\hline
\multicolumn{3}{|c|} 
{
\text{\vspace*{0.5em}The following items assume } \bbF = \bbR, 
\text{ the field of real numbers\vspace*{0.5em}}
}\\
\hline
A(\bz) \unrhd B(\bz)&\text{ says }&a_\bu \ge b_\bu\text{ for all }\bu\\
\bA(\bz) \unrhd \bB(\bz)&\text{ says }&A_i(\bz) \unrhd B_i(\bz)\text{ for all } i\\
\bA(\bz)> \bzero&\text{ says }&A_i(\bz) \neq 0, \text{ for all } i\\
\dom[\bz] &=&\{A(\bz) \in \bbR[[\bz]] : A(\bz) \unrhd \bzero\}\\
\dom_0[\bz] &=&\{A(\bz) \in \dom[\bz] : A(\bzero) = 0\}\\
\dom_{J0}[x,\by]&=& \{\bG(x,\by) \in \dom_0[x,\by]^k : J_\bG(0,\bzero) = \bzero\}, \text{ where } \by = y_1,\ldots,y_k\\
\hline
\end{array}
$$
}
%
For $k\ge 1$, the set $\bbF[[x]]^k$ becomes a complete
 metric space when equipped with the metric
$$
d\big(\bA(x),\bB(x)\big)\  :=\  
\begin{cases}
2^{-\min\,\ldegree\big(A_i(x) -B_i(x)\,:\, 1 \le i\le k\big)} &\text{if } \bA(x)\neq \bB(x)\\
0&\text{if }\bA(x) = \bB(x).
\end{cases}
$$
 One has $d\big(\bA_n(x),\bB_n(x)\big) \rightarrow 0$ as $n\rightarrow \infty$ iff
for all $m\ge 0$ there is an $N\ge 0$ such that $[x^{\le m}]\bA_n(x) = [x^{\le m}]\bB_n(x)$
for $n\ge N$; that is, for $n$ sufficiently large, the corresponding coordinates of $\bA_n$ and $\bB_n$
agree on their first $m+1$ coefficients.
The subset $\bbF[[x]]_0^k$ of $\bbF[[x]]^k$ is, with the same metric, also a complete metric
space. 

Let $k\ge 1$ be given, and let $\by := y_1,\ldots,y_k$. 
Given a $k$-tuple of formal power series 
$\bG(x,\by)\in\bbF[[x,\by]]_0^k$, 
and given $\bA(x)\in\bbF[[x]]_0^k$,
the composition  $\bG(x,\bA(x))$ is a well-defined member of $\bbF[[x]]_0^k$ 
if $\bG(x,\bzero) = \bzero$.
 (This is a sufficient, but not necessary condition.)
Such a $\bG(x,\by)$ can be viewed as a mapping from $\bbF[[x]]_0^k$ to itself, a mapping 
whose $n$-fold composition
with itself will be expressed by  $\bG^{(n)}(x,\by)$, a well-defined member of $\bbF[[x,\by]]_0^k$.
More precisely, 
\begin{eqnarray*}
\bG^{(0)}(x,\by) &=& \by,\\
\bG^{(n+1)}(x,\by) &=& \bG(x,\bG^{(n)}(x,\by)).
\end{eqnarray*}
The power series in the $i$th coordinate of $\bG^{(n)}(x,\by)$ will be denoted by $\bG_{\ i}^{(n)}(x,\by)$,
that is,
\begin{eqnarray*}
\bG^{(n)}(x,\by) & = & \big( \bG_{\ 1}^{(n)}(x,\by), \ldots, \bG_{\ k}^{(n)}(x,\by) \big).
\end{eqnarray*}

The basic results on existence and uniqueness of
solutions to systems hold in a quite general setting. When one wants to analyze the solutions
or the spectra in more detail, it becomes beneficial to use the real field $\bbR$.
%

\begin{proposition} \label{gen GJ=0}
Let $\bG(x,\by)\in\bbF[[x,\by]]^k$. 
If 
\begin{thlist}
\item
$\bG(0,\bzero) = \bzero$ and

\item
$J_\bG(0,\bzero) = \bzero$
\end{thlist}
then the equational system $\by = \bG(x,\by)$ 
\begin{thlist}
\item [i] has a unique solution $\bT(x)$ in $\bbF[[x]]_0^k$, 

\item [ii]
$\bT(x)$ satisfies the initial condition
$\bT(0) = \bzero$, and,

\item [iii]
 for any $\bA(x) \in \bbF[[x]]_0^k$, one has $($in the aforementioned complete metric space$)$
$$
\bT(x) = \lim_{n\rightarrow \infty} \bG^{(n)}(x,\bA(x)).
$$
\end{thlist}
If, furthermore,
\begin{thlist}
\item[c]
$\bbF =\bbR$,
\end{thlist}
then
\begin{thlist}
\item[iv]
$\bG(x,\by)\unrhd \bzero\ \Rightarrow\ \bT(x)\unrhd \bzero$.
\end{thlist}
\end{proposition}
\begin{proof}
For $\bA(x), \bB(x) \in \bbF[[x]]_0^k$ the hypotheses guarantee that
$$
[x^{\le n}]\bA(x) = [x^{\le n}]\bB(x)\quad \Rightarrow\quad 
[x^{\le n+1}]\bG(x,\bA(x)) = [x^{\le n+1}]\bG(x,\bA(x)).
$$
This implies that $\bG(x,\by)$ is a contraction mapping on the complete 
metric space $\dom_0[x]$, consequently  (i)--(iii) follow. Item (iv) follows
from (iii).
\end{proof}

\begin{definition} 
Given a power series $T(x)$, let $T = \Spec(T(x))$, the {\em spectrum} of $T(x)$,
 be the support of the sequence $t(n)$ of coefficients of $T(x)$,
that is, $T := \{n\ge 0 : t(n) \neq 0\}$. Extend the definition of spectrum to $k$-tuples $\bT(x)$
of power series by
$\bT = \Spec(\bT(x)) := (T_1,\ldots,T_k)$.
\end{definition}

 For $T(x) \in\bbF[[x]]$ let $\fm :=  \fm(T)$ and $\fq :=  \fq(T)$, as in
Definition \ref{defn per param}.  It is quite easy to see that the following hold:
 \begin{thlist}
 \item
 $x^{\fm}$ is the largest power of $x$ dividing $T(x)$, that is, 
 $\fm$ is the smallest index $n$
 such that $t(n)\neq 0$, 
 \item
 $x^{\fq}$ is the largest power of $x$ such that for $n\ge 0$,
 $t(n)\neq 0$ implies $\fq\, |\, n-\fm$.
 \item
There is a $($unique$)$ power series $V(x)\in\bbF[x]$
such that 
	$T(x) = x^\fm V(x^\fq)$. 
One has $\gcd(V)=1$.
\item
Suppose $\bbF = \bbR$ and $T(x)\unrhd 0$. If the radius of convergence
$\rho$ of $T(x)$ is in $(0,\infty)$ then the dominant singularities of $T(x)$
are $\rho\cdot \omega^j$, $j=0,\ldots, \fq-1$, where $\omega$ is a primitive
$\fq$th root of unity.
\end{thlist}
Under favorable conditions --- such as those encountered in \cite{BBY2006}, a study of
non-linear single equation systems $y = G(x,y)$ with solution $T(x)$ --- the spectrum $T$ of $T(x)$
is the union of a finite set and an arithmetical progression, and the coefficients $t(n)$
of $T(x)$ have `nice' asymptotics for $n$ on this spectrum. 
It would be an important achievement to show that any
system $\by = \bG(x,\by)$ built from standard components would have a solution $\bT(x)$
with the $T_i(x)$ exhibiting the positive features just described.  A first, and very modest 
step in this direction, is to show that such systems have spectra of the appropriate kind, 
namely eventually periodic spectra. This positive first step is achieved in Section 
\ref{Gen Sys}.

\subsection{Non-Negative Power Series and Elementary Systems}
A power series $A(\bz)\in\bbR[[\bz]]$ is \emph{non-negative} if $A(\bz)\unrhd 0$,
that is, each coefficient $a_\bu$ is non-negative.  $\bA(\bz)\in\bbR[[\bz]]^m$
is {non-negative} if each $A_i(\bz)$ is non-negative.
A system $\by = \bG(x,\by)\in\bbR[[x,\by]]^k$ is {non-negative}
if $\bG(x,\by)$ is non-negative. 

A non-negative power series $G(x,\by)$ can be expressed in the form
$$
\sum_{\bu \in \bbN^k} G_\bu(x)\cdot \by^\bu,
$$
where $\by^\bu$ is the monomial $y_1^{u_1}\cdots y_k ^{u_k}$.
A non-negative system $\by = \bG(x,\by)$ is \emph{elementary} iff
$\bG(0,\bzero)=J_\bG(0,\bzero)=\bzero$; this condition is easily seen
to be equivalent to requiring:
 for $\bu\in\bbN^k$
and $1\le i \le k$,
$$
G_{i,\bu}(0) \neq 0\ \Rightarrow\ \sum_{j=1}^k u_j \ge 2.
$$

When working with non-negative
power series, the $\Spec$ operator acts like a homomorphism, as the
next lemma shows. This allows one to convert equational specifications,
or equational systems defining generating functions, into equational
systems about spectra.

%
\begin{lemma} \label{spec hom}
Let $c> 0$ and let $A(x),A_i(x), B(x)\in \bbR[[x]]$ be non-negative power series. Then
\begin{thlist}
\item
$\Spec\big(c \cdot  A(x)\big) = A$

\item
$\Spec\big(A(x) + B(x)\big) = A \, \cup \,  B$

\item
$\Spec\Big(\sum_i A_i(x)\Big)\   =  \  \bigcup_i A_i$, 
\quad 
provided $\sum_i A_i(x)\in\bbR[x]$

\item
$\Spec\big(A(x)  \cdot  B(x)\big) = A + B$

\item
$\Spec\big(A(x) \circ B(x)\big) = A\star B$, provided $B(x) \in \bbR[x]_0$.
\end{thlist}
\end{lemma}

\begin{proof}
The first four cases (scalar multiplication, addition and Cauchy product) 
are straight-forward, as is  composition: 
\begin{eqnarray*}
   \Spec\big(A(x) \circ B(x)\big) & = & \Spec \sum_{i\ge 1} a(i) B(x)^i \ 
   =  \  \bigcup_{i\in A} \Spec\,\big( B(x)^i\big)\\
   & =&  \bigcup_{i\in A} i\star B \   =  \ A\star B.
\end{eqnarray*}
\end{proof}
%

One defines the dependency digraph $D_\bG$ for a system $\by = \bG(x,\by)$ parallel to
the way one defines it for a system of set-equations $\bY =  \bGamma(\bY)$, namely
$i\rightarrow j$ iff $G_i(x,\by)$ depends on $y_j$. 

\begin{lemma} [Tests for eventually dependent]
Given a non-negative system $\by = \bG(x,\by)$,  the following are equivalent:
\begin{thlist}
\item
$i\rightarrow^+ j$

\item
there is an $m \in \{1,\ldots,k\}$ such that
the $(i,j)$ entry of ${J_\bG}(x,\by)^m$ is not $0$

\item
the $(i,j)$ entry of $\sum_{m=1}^k {J_\bG}(x,\by)^m$ is not $0$.
\end{thlist}
\end{lemma}
%

In practice one only works with systems that have a connected dependency digraph. 
Otherwise the
system trivially breaks up into several independent subsystems.
There has been considerable interest in irreducible systems,
 where every
$y_i$ eventually depends on every $y_j$. 
Such systems behave similarly to one-equation
systems.
However, even some non-negative irreducible systems $\by = \bG(x,\by)$ can be easily
decomposed 
into several independent subsystems --- this will happen precisely when ${J_\bG}(x,\by)^k$
has some zero entries. If not, then ${J_\bG}(x,\by)^k > \bzero$, which is precisely the case 
when the matrix ${J_\bG}(x,\by)^k$ is {\em primitive} --- this is equivalent to the system
being {\em aperiodic} and {\em irreducible}. $($See, for example, \cite{FlSe2009}.$)$
Awareness of the possibility of decomposing irreducible systems is important for practical
computational work. The next result is our main theorem on power series systems.

\begin{theorem}  \label{Thm PSeries}
For an elementary system $\by = \bG(x,\by)$
 the following hold:
\begin{thlist}
\item
The system  has a unique solution 
$\bT(x)$ in  $\bbR[[x]]_0^k$.

\item
$\bT(x)\unrhd \bzero$, that is, the coefficients of each $T_i(x)$ are non-negative.

\item
 $\displaystyle \bT(x) = \bG^{(\infty)} (x,\bA(x)) := \lim_{n\rightarrow\infty}  \bG^{(n)} (x,\bA(x))$,
  for any $\bA(x)\in\bbR[[x]]$
 satisfying $\bA(0)=\bzero$.

\item
The $k$-tuple $\bT$ of spectra $T_i$ is the unique solution to the 
elementary system of set-equations
$\bY = \bGamma(\bY)$ where
$$
 \bGamma(\bY)\  := \ \bigcup_{\bu\in\bbN^k} \bG_{\bu} + \bu \star \bY.
$$

\item
$\displaystyle
\bT \ =\ \bGamma^{(\infty)}(\bA)\ :=\ 
\lim_{n\rightarrow \infty}\bGamma^{(n)}(\bA), \ \text{for any } \bA\in\Su(\bbP)^k.
$

\item
$T_i(x) = 0\ $ iff $\ G_{\ i}^{(k)}(x,\bzero) = 0\ $ iff $\ T_i = \O\ $ iff 
$\ \Gamma_{\ i}^{(k)}(\bempty) = \O\ $ iff $\ \fm_i = \infty$.

\end{thlist}
Now we assume that the system has been reduced by eliminating all
$y_i$ for which $T_i(x)=0$.
\begin{thlist}

\item[g]
 $[i] \neq\O$ implies $T_i$ is periodic. If also there is a $j\in[i]$ such that
 for some $\bu\in\bbN^k$ one has
 $G_{j,\bu}\neq \O$ and $\sum\{u_\ell : \ell \in [i]\} \ge 2$, then $T_i$ is the union of
 a finite set with a single arithmetical progression.

\item[h]
 If $[i] = \O$ and the i\:{\em th} equation can be written in the form
 $$
 Y_i\ :=\ P_i \  +\ \bigcup_{Q \in \fQ_i} \sum_{j=1}^k Q_j\star Y_j,
 $$
 with $P_i$ [eventually] periodic, and with  $\fQ_i$ a finite set of $k$-tuples $Q = (Q_1,\ldots,Q_k)$ of 
 [eventually] periodic subsets $Q_j$ of $\bbN$, and if
 for $i \rightarrow j$ one has $T_j$ being [eventually] periodic, then
 $T_i$ is [eventually] periodic.
%
\item[i]
The periodicity parameters $\bfm,\bfq$ of $\bT$ can be found from
$\bGamma^{(k)}(\bempty) $ and
the 
$\bG_\bu$ via the formulas
\begin{eqnarray}
 	\fm_i :=  \fm_i(T_i) 
		&=& \min\Big(\Gamma_{\ i}^{(k)}(\bempty)\Big) \label{m formula2}\\
	\fq_i := \fq(T_i) 
		&=&\gcd \bigg(\bigcup_{i\raStar  j}
 		\bigcup_{\bu\in\bbN^k} \Big( G_{j,\bu} + \bu\star\bfm  -\fm_j \Big) \bigg). \label{q formula2}
 \end{eqnarray}

 \item[j]
$\fq_i\, |\, \fq_j$ whenever $i\rightarrow j$.
\end{thlist}
\end{theorem}

\begin{proof}
Items (a)--(c) are immediate from Proposition \ref{gen GJ=0}.
For (d) simply apply $\Spec$ to both sides of 
$\bT(x) = \bG(x,\bT(x))$. For  (e)--(j) note that the hypotheses
of  the theorem imply that $\bGamma(\bY)$ satisfies the
hypotheses of Theorem \ref{ThmSetEqn}, so one can use
the formulas \eqref{m formula} and \eqref{q formula}.

\end{proof}

Systems that arise in combinatorial problems are invariably reduced since the 
solution gives generating functions for non-empty classes of objects. However if
one should encounter a non-reduced elementary polynomial system $\by = \bG(x,\by)$, 
Theorem \ref{Thm PSeries}\,(f) provides an efficient way to determine which
of the solution components $T_i(x)$ will be 0, namely let 
$\mu$ map any member $A(x)\in\dom_0[x]$ to its lowest degree term, setting the 
coefficient to 1;  extend
this to $\dom_0[x]^k$ coordinate-wise. Then
$$
T_i(x)=0\quad\text{iff}\quad
\Big((\mu \circ \bG)^{(k)}(x,\bzero)\Big)_i = 0. 
$$

%
\subsection{Periodicity Results for Linear Systems}
Irreducible linear equations $y = G_0(x) + G_1(x)y$ do not, in general, have the property that
the spectrum $T$ is eventually an arithmetical progression. For example, let $y=T(x)$ be 
the power series solution to
\[
	y \   =  \ x + x^2 + x^3y.
\]
The periodicity parameters of $T$ are $\fm=1$, $ \fq=1$, 
$\fp=3$, and $\fc=1$.
$T$ is readily seen to be 
$$
	\{3n+1 : n\ge 0\}\cup\{3n+2: n\ge 0\},
$$
and the set of periods of $T$ is the same as the set of eventual periods of $T$, namely $3\cdot \bbN$.

The spectrum of a 1-equation elementary linear system has a particularly simple expression.
%
%
\begin{proposition}\label{lin sys}
Given a 1-equation elementary linear system
$$
	y \   =  \  G(x,\by) \ :=\ G_0(x) + G_1(x)\cdot y,
$$
the solution is 
$$
T(x)\ =\ \Big(\sum_{n\ge 0}G_1(x)^n\Big) \cdot G_0(x),
$$
the spectral equation is
$$
	Y \   =  \  \Gamma(Y) \ :=\ G_0 \cup (G_1+Y),
$$
and the spectrum is
$$
T\ =\ \Big(\bigcup_{n\ge 0} n\star G_1\Big) + G_0\ =\ G_0 + \bbN\star G_1.
$$
Thus $\fm(T) =  \min(G_0)$ and $\fq(T) = \gcd\big((G_0 - \fm(T)) \cup G_1\big) $.
\end{proposition}
%
The proof of the proposition is straightforward. 
From the form of the solution for $T$ one sees that every periodic subset of $\bbP$ 
is the spectrum of the solution to some 1-equation linear system.

The next two examples, of linear systems, are cornerstones in the study of 
systems.

%
\begin{example}[Postage Stamp Problem]  \label{postage}
The {\em postage stamp problem} $($an equivalent version is called the {\em coin change
problem}$)$ asks for the amounts of postage one can put on a package if one has stamps in
denominations $d_1,\ldots,d_r$. 
With $D=  \{d_1,\ldots,d_r\}$ the set of denominations of the stamps, let 
$D(x) = \sum_{i=1}^r x^{d_i}$. 
Then the postage stamp problem has the generating function $S(x)$ $($with $s(n)$ giving the
number of ways to realize the postal amount $n$$)$ being the solution to the elementary linear recursion
\[
   y\   =  \ D(x) + D(x) \cdot y.
\]
The spectrum $S$ is the solution to the set-equation
\[
   Y = D  \, \cup \,  (D + Y),
\]
which, by Proposition \ref{lin sys} is 
$S= \bbP\star D$. By Lemma \ref{lin comb}, 
$S$ is periodic, $\fq = \fp = \gcd(S) = \gcd(D)$, and
$S = S\big|_{<\fc} \cup (\fc + \fq\cdot \bbN)$, 
where $\fc := \fc(S)$, etc.\footnote
{The number $\gamma(D) := \fc(S)$ is called 
{\em the conductor of} $D$ by Wilf (see \cite{Wilf1994}, $\S$3.15.).
$\gamma(D)-1$ is called the \textit{Frobenius number}, and the problem of finding it
is called the Frobenius Problem (or Coin Problem). The problem can easily be reduced to
the case that $\gcd(D) = 1$, in which case every number $\ge \gamma(D)$ is in $\bbN\star D$, 
but $\gamma(D) -1 \notin \bbN\star D$.
For $D$ a finite set of positive integers, considerable effort has been devoted to finding a formula 
for $\gamma(D)$ for $D$ with few elements.  
The only known closed forms are for $D$ with 1, 2 or 3 elements. 
For $D = \{b_1,b_2\}$ with 2 co-prime elements , the solution is 
$\gamma(D) = (b_1-1)(b_2-1)$, found by
Sylvester in 1884. Finding $\gamma(D)$ is known to be NP-hard.}
 \end{example}

 %
 \begin{example}[Paths in Labelled Digraphs]  \label{paths}
The objective in this example is to find the set of lengths of the paths
 going from vertex 1 to vertex 4
 in the labelled digraph in Fig.~\ref{digraph pic}.
 %
 \begin{figure}[h] 
   \centering
   \includegraphics[width=2in]{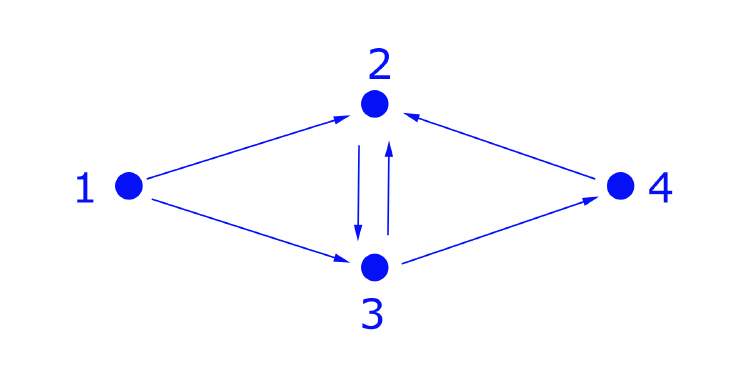} 
   \caption{A Labelled Digraph}
   \label{digraph pic}
\end{figure}

For $1\le i \le 4$, let $L_i(x) = \sum_{n\ge 1} \ell(n) x^n$ be the generating function for the 
lengths of paths going from vertex i to vertex 4, that is, 
$\ell_i(n)$ counts the number of paths of length $n$ from vertex $i$ to vertex $4$. 
 Then  $\by = \bL(x)$ satisfies the following system $\by = \bG(x,\by)$:
 \begin{eqnarray*}
    y_1 & = &  x \cdot \big(y_2 +y_3\big)\\
    y_2 & = &  x \cdot y_3\\
    y_3 & = & x \cdot \big(y_2 + 1+y_4\big)\\
    y_4 & = & x \cdot y_2.
\end{eqnarray*}
One has $\bG(x,\by)\unrhd 0$ and $\bG(0,\bzero) = J_\bG(0,\bzero) = \bzero$, 
so the system is elementary.
The associated elementary spectral system $\bY=\bGamma(\bY)$ is:
 \begin{eqnarray*}
    Y_1 & = &  1 + (Y_2 \cup Y_3)\\
    Y_2 & = &  1 + Y_3\\
    Y_3 & = &  1 \cup  ( 1 + (Y_2 \cup Y_4))\\
    Y_4 & = & 1 + Y_2.
\end{eqnarray*}
To calculate the $\fm_i$ and $\fq_i$ for this system, first
$$
\bGamma(\bempty)\ =\ 
\left [ 
\begin{array}{c}
\O\\
\O\\
1\\
\O
\end{array}
\right]
\quad
\bGamma^{(2)}(\bempty)\ =\ 
\left [ 
\begin{array}{c}
2\\
2\\
1\\
\O
\end{array}
\right]
\quad
\bGamma^{(3)}(\bempty)\ =\ 
\left [ 
\begin{array}{c}
\{2,3\}\\
2\\
\{1,3\}\\
3
\end{array}
\right]
\quad
\bGamma^{(4)}(\bempty)\ =\ 
\left [ 
\begin{array}{c}
\{2,3,4\}\\
\{2,4\}\\
\{1,3,4\}\\
3
\end{array}
\right],
$$
thus, by \eqref{m formula2}, $\fm = (2,2,1,3)$. For such a simple example one also easily finds the
$\fm_i$ by inspection --- $\fm_i$ is the length of the shortest path in 
Fig.~\ref{digraph pic} from vertex $i$ to vertex $4$.

To calculate the $\fq_i$ let 
$$
S_j \ :=\ \bigcup_\bu G_{j,\bu} + \bfm\star\bu - \fm_j,\quad \text{for }1 \le j \le 4.
$$
Then $S_1 = \{ 0,1\}$, $S_2 = \{0\}$, $S_3 = \{0,2,3\}$, and $S_4=\{0\}$. 
The digraph in Fig.~\ref{digraph pic} is, conveniently, also the dependency
digraph of the system, and
$\{2,3,4\}$ is a strong component. From \eqref{q formula2},
 $\fq_i = \gcd\bigcup_{i\raStar j} S_j$,
so 
$\fq_1 = \gcd\big(S_1\cup S_2\cup S_3\cup S_4\big) = \gcd\{0,1,2,3\} = 1$, 
and $\fq_2=\fq_3 = \fq_4 = \gcd\big(S_2\cup S_3\cup S_4\big) = \gcd\{0,2,3\} = 1$.
\end{example}
%

%
\subsection{Relaxing the Conditions on $\bG(x,\by)$}
Recall that
a power series system $\by = \bG(x,\by)$ is elementary if 
(i) $\bG(x,\by)\unrhd \bzero$, (ii) $\bG(0,\bzero) = \bzero$ and 
(iii) $J_\bG(0,\bzero) = \bzero$.

The `elementary system' requirement of Theorem \ref{Thm PSeries} is usually 
true for power series systems $\by = \bG(x,\by)$ arising in combinatorics --- see, for example,
the book \cite{FlSe2009} of Flajolet and Sedgewick, where most of the examples are such
that $x$ is a factor of $\bG(x,\by)$, a property of $\bG(x,\by)$ which immediately guarantees 
that the second and third of the three conditions holds. 
The second condition, $\bG(0,\bzero)=\bzero$,
is essential if the solution $\bT(x)$ provides generating functions $T_i(x)$ for 
combinatorial classes $\cT_i$ since, in these cases, $T_i\subseteq \bbP$,
so $0\notin T_i$, for an $i$.

Dropping the first requirement, that  $\bG(x,\by) \unrhd \bzero$,
leads to a difficult area of research where little is known, even with a single equation
$y = G(x,y)$ --- see the final sections of \cite{BBY2006} for several remarks on the difficulties
mixed signs in $G(x,y)$ pose when trying to determine the asymptotics of the coefficients 
$t(n)$ of a solution $y = T(x)$. 
Such mixed sign situations can arise naturally, for example when dealing
with the construction $\Set$, which forms subsets of a given set of objects. 
The method 
developed in this paper for studying the spectra of the solutions $T_i(x)$ of a system 
$\by = \bG(x,\by)$ very much depends on $\bG(x,\by) \unrhd \bzero$, in particular,
claiming that $\Spec\big(\bG_\bu(x)\cdot \bT(x)^\bu\big)$ is equal to 
$\bG_\bu + \bu\star \bT$. This equality can fail with mixed signs, for example, 
the spectrum of $(1-x)\cdot (1+x+x^2)$ is not
the same as $\Spec(1-x) + \Spec(1+x+x^2)$.

Thus the discussion regarding strengthening the results of the previous sections
will be limited 
to dropping the third requirement, that $J_\bG(0,\bzero) = \bzero$. 
This simply means that linear $\by$-terms with constant coefficients are permitted to 
appear in the $G_i(x,\by)$, in which case a number of new possibilities 
can arise when classifying the solutions of such systems:
\begin{thlist}
\item
There may be no (formal power series) solution, for example,
$y = x + y$.

\item
There may be a solution, but not $\unrhd\, \bzero$, for example,
$y = x + 2y$.

\item
There may be infinitely many solutions, for example, 
$y_1 = y_2$, $y_2 = y_1$.
\end{thlist}

One can express the system
$ \by = \bG(x,\by) $ 
as
$$
   \by\   =  \  \bG(x,\bzero) + J_\bG(0,\bzero)\cdot\by + \bH(x,\by),
$$
where 
$$
   \bH(x,\by)\   =  \  \sum_{i=1}^k y_i \cdot \bH_i(x,\by)
$$
with each 
$\bH_i(x,\by)  \in  \bbR[[x,\by]]_0^k$. 

The obvious approach to such a system
with $J_\bG(0,\bzero) \neq \bzero$ 
is to write it in the form
$$
\big(I - J_\bG(0,\bzero)\big)\cdot \by \ =\ \bG(x,0) + \bH(x,\by)
$$
and solve for $\by$. 

%
\begin{definition}[of $\widehat{\bG}$] \label{widehat G}
Given $\bG(x,\by)\unrhd \bzero$ with $\bG(0,\bzero)=\bzero$,
if the matrix $I-J_\bG(0,\bzero)$ has an inverse that is non-negative
then
let
 $$
	 \widehat{\bG}(x,\by)\ := \ 
	 \Big(I - J_\bG(0,\bzero)\Big)^{-1}
	\cdot \Big(\bG(x,\bzero) + \bH(x,\by)\Big).
$$
\end{definition}
%

Given a non-negative square matrix $M$, let $\Lambda(M)$ denote the largest real 
eigenvalue of $M$. 
(Note: From the Perron-Frobenius theory we know that a non-negative square matrix $M$ 
has a non-negative real eigenvalue, 
hence there is indeed a largest real eigenvalue $\Lambda(M)$, it is $\ge 0$,
and $\Lambda(M)$ has a non-negative eigenvector.)

%
\begin{theorem}\label{Consis Thm}
Let $\bG(x,\by)\in\bbR[[x,\by]]^k$ satisfy the two conditions
$$
\bG(x,\by)\unrhd \bzero,\ \text{ and }
\bG(0,\bzero) = \bzero.
$$
\begin{thlist}
\item
Suppose $ I - J_\bG(0,\bzero)$ has a non-negative inverse.

\subitem {\em (i)}
The system $\by = \widehat{\bG}(x,\by)$ is equivalent to the system $\by =\bG(x,\by)$,
that is, they have the same solutions $($but not necessarily the same dependency digraph$)$.

\subitem{\em (ii)}
$\widehat{\bG}(x,\by)$ is an elementary system.

\subitem {\em (iii)}
Consequently 
$\bT(x) := \widehat{\bG}^{(\infty)}(x,\bzero)$ is the unique solution in $\bbR[[x]]_0$ 
of $\by =\bG(x,\by)$ as well as of $\by = \widehat{\bG}(x,\by)$. The periodicity
properties of $\bT(x)$ are as stated in Theorem \ref{Thm PSeries}.

\item
 Suppose that $\bG^{(k)}(x,\bzero) > \bzero$,  that is, the associated 
 system $\bY = \bgam(\bY)$ of set equations is reduced. 
 Then the following are equivalent:

\subitem {\em (i)}  $I - J_\bG(0,\bzero)$ has a non-negative inverse.

\subitem {\em (ii)} $\by =\bG(x,\by)$ has a solution $\bT(x)\in\dom_0[x]$.

\subitem {\em (iii)} $\Lambda\big(J_\bG(0,\bzero)\big) < 1$.
\end{thlist}
\end{theorem}

\begin{proof}
(a): Given that $I-J_\bG(0,\bzero)$ has a non-negative inverse, one can transform either 
of $\by =\bG(x,\by)$ and $\by = \widehat{\bG}(x,\by)$ into the other by simple operations
that preserve solutions. It is routine to check that $\widehat{\bG}(x,\by)$ is an elementary system.

(b): Assume $\bG^{(k)}(x,\bzero) > \bzero$. 
(i) $\Rightarrow$ (ii) follows from (a).
If (ii) holds then 
$$
 \bT(x)\   =  \  \bG^{(k)}\big(x,\bT(x)\big)
	  \  \unrhd\   \bG^{(k)}(x,\bzero)\   >  \  \bzero.
$$
Let $\bv\ge 0$ be a left eigenvector of 
$ \Lambda\big(J_{\bG^{(k)}}(0,\bzero)\big) $.
From
$$
   {\bT(x)}\   =  \  \bG^{(k)}(x,\bzero) \,+\, J_{\bG^{(k)}}(0,\bzero)\cdot{\bT(x)}\, +\, 
	\widetilde{\bH}\big(x,{\bT(x)}\big),
$$
one has
\begin{equation}\label{v dot}
	\bv\cdot\bT(x)\   =  \  \bv\cdot \bG^{(k)}(x,\bzero)\, +\,
 	 \Lambda\big(J_{\bG^{(k)}}(0,\bzero)\big)\cdot\bv\cdot \bT(x)\, + \,
	\bv\cdot\widetilde{\bH}\big(x,{\bT(x)}\big).
\end{equation}
Since $\bT(x) > \bzero$, one has 
$\bv\cdot\bT(x)$ 
and 
$
  \bv\cdot \bG^{(k)}(x,\bzero)+ \widetilde{\bH}\big(x,\bT(x)\big)
$ 
are non-zero power series with non-negative coefficients, 
consequently \eqref{v dot} implies
$
   \Lambda\big(J_{\bG^{(k)}}(0,\bzero)\big)  < 1.
$
From $J_{\bG^{(k)}}(0,\bzero) = J_{\bG}(0,\bzero)^k$ it follows that
$\Big(\Lambda\big(J_{\bG}(0,\bzero)\big)\Big)^k$ is an eigenvalue of 
$ J_{\bG^{(k)}}(0,\bzero)$, and thus also $<1$. But this clearly implies
$\Lambda\big(J_{\bG}(0,\bzero)\big) <1$, so (ii) $\Rightarrow$ (iii).

If (iii) holds, then by Neumann's expansion theorem  (see \cite{Or1987}, p.~201),
 one knows that
 $I-J_\bG(0,\bzero)$ has an inverse, and 
 $( I-J_\bG(0,\bzero)^{-1} = \sum_{n\ge 0} J_\bG(0,\bzero)^n$, a 
 non-negative matrix. Thus (iii) $\Rightarrow$ (i).
\end{proof}
%

The condition $\bG^{(k)}(x,\bzero) > \bzero$ is the norm for power series systems in
combinatorics since the $T_i(x)$ in the solution of $\by = \bG(x,\by)$ are generating
functions for non-empty classes $\cT_i$.

It turns out (but will not be proved here) that for the calculation of the 
$\fm_i$ and $\fq_i$ one can use the formulas \eqref{m formula2} and \eqref{q formula2}
of Theorem \ref{Thm PSeries} with the original system $\by = \bG(x,\by)$ as well
as with the derived system  $\by = \widehat{\bG}(x,\by)$. It can be useful to note that if 
the two hypotheses of Theorem \ref{Consis Thm} hold,
then the condition $\bG^{(k)}(x,\bzero) > \bzero$ is equivalent to requiring that
$\bG^{(j)}(x,\bzero) > \bzero$ hold for some $j$, $1\le j\le k$.
  
\begin{remark} 
The uniqueness of solutions $\bT(x)$ in $\bbR[[x]]_0^k$ for power series systems 
$\by  = \bG(x,\by)$ satisfying the two hypotheses of Theorem \ref{Consis Thm} does not 
in general carry over to 
the associated spectral systems $\bY = \bGamma(\bY)$
when $J_\bG(0,\bzero) \neq \bzero$. For example, consider the consistent single equation 
system $y=G(x,y)$ where $G(x,y) = x^2 +(1/2)y + xy$. The spectral system $Y = \gam(Y)$
is $Y = 2\cup Y \cup (1+Y)$, which has three solutions  $\bbN$, $1+\bbN$, and $2+\bbN$. 
The elementary system
$y = \widehat{G}(x,y)$ is $y = 2x^2 + 2xy$; its spectral system is $Y = 2 \cup (1+Y)$,
which has the unique solution $2+\bbN$.
\end{remark}

\subsection{A Non-linear Polynomial System}
The following simple example uses all the tools developed so far.
\begin{example} \label{blue and red}
 Consider the class $\cT$ of planar trees with blue and red colored nodes, defined by 
 the conditions: 
 \begin{thlist}
 \item [i] 
    every blue node that is not a leaf has exactly three subnodes, but not all 
    of the same color;
 \item [ii]
   and every red node that is not a leaf has exactly two subnodes. 
 \end{thlist}
 Let $\cB$ be the collection of trees in $\cT$ with the root colored blue, and likewise define
 $\cR$ for the red-colored roots. Then, letting $\bullet_B$ be a blue-colored node and
 $\bullet_R$ a red-colored node, one has the equational specification
\begin{eqnarray*}
   \cB & = &  \{\bullet_B\}\ \cup\ 
      \frac{\bullet_B} { \cB +\cR + \cR}\  \cup \ 
   \frac{\bullet_B} { \cR +\cB + \cR}\  \cup \ 
   \frac{\bullet_B} { \cR +\cR + \cB}\  \cup \ \\
      &&\qquad
   \frac{\bullet_B} { \cB +\cB + \cR}\  \cup \ 
   \frac{\bullet_B} { \cB +\cR + \cB}\  \cup \ 
   \frac{\bullet_B} { \cR +\cB + \cB}\\
    \cR & = &  \{\bullet_R\}\ \cup\ \frac{\bullet_R} {\cT + \cT}\\
    \cT&=&\cB + \cR.
   \end{eqnarray*}
The three generating functions, $B(x)$ for $\cB$,  $R(x)$ for $\cR$, and
$T(x)$ for  $\cT$, are related by the system of equations:
\begin{eqnarray*}
   B(x) & = &  x \  +\  3x \cdot  B(x) \cdot  R(x)^2\   + \ 3 x \cdot  B(x)^2 \cdot  R(x)\\
   R(x) & = & x\  +\  x \cdot  T(x)^2\\
   T(x) & = & B(x)\  +\  R(x).
\end{eqnarray*}
Thus  $(B(x),R(x),T(x))$ gives a solution for $(y_1,y_2,y_3)$ in the system of polynomial equations:
\begin{eqnarray*}
   y_1 & = &  x \  +\  3x \cdot  y_1 \cdot  y_2^2\   + \ 3 x \cdot  y_1^2 \cdot  y_2\\
   y_2 & = & x\  +\  x \cdot  y_3^2\\
  y_3 & = & y_1\  +\  y_2.
\end{eqnarray*}
The spectra $B,R,T$ are related by the set-equations
\begin{eqnarray*}
   B & = &  1 \  \cup \  (1 +  B +  2\star R)\   \cup \ (1 +  2\star B +  R)\\
  R & = & 1\  \cup\  (1 +  2\star T)\\
 T & = & B\  \cup\  R,
\end{eqnarray*}
so $(B,R,T)$ is a solution to the system of set-equations
\begin{eqnarray*}
   Y_1 & = &  1 \  \cup\  1 +  (Y_1 + 2\star Y_2)\   \cup \ (1 +  2\star Y_1 +  Y_2)\\
  Y_2 & = & 1\  \cup\  (1 + 2\star  Y_3)\\
 Y_3 & = & Y_1\  \cup\  Y_2.
\end{eqnarray*}

Next,
$$
\bG(x,y_1,y_2,y_3)\ =\ 
\left[
\begin{array}{c }
x \, +\, 3x \cdot  y_1 \cdot  y_2^2\,  + \,3 x \cdot  y_1^2 \cdot y_2\\
x\, +\, x \cdot  y_3^2\\
y_1\, +\, y_2
\end{array}
\right],
$$
so
$$
\bG^{(2)}(x,0,0,0)\ =\ 
\left[
\begin{array}{c }
 6x^4\, +\, x\\
 x\\
 2x
\end{array} 
\right]\ >\  \bzero.
$$
This implies $\bG^{(k)}(x,\bzero) > \bzero$, where $k=3$.

The Jacobian matrix $J_\bG(x,\by)$  is
$$
J_\bG(x,y_1,y_2,y_3)\ =\ 
\left[
\begin{array}{c  c  c}
3xy_2^2\,+\,6xy_1y_2 & 6xy_1y_2\,+\,3xy_1^2 & 0\\
0 &0 &2xy_3\\
1 &1 &0
\end{array}
\right]
$$
so 
$$
J_\bG(x,0,0,0)\ =\ 
\left[
\begin{array}{c  c  c}
0 & 0 & 0\\
0 &0 &0\\
1 &1 &0
\end{array}
\right].
$$
The eigenvalues of $J_\bG(0,0,0,0)$ are the roots
of 
$\det\big(\lambda I - J_\bG(0,0,0,0)\big) = 0$, 
that is, 
$\lambda^3 =0$. Thus $\Lambda\big(J_\bG(0,0,0,0)\big) = 0 < 1$, 
so the system 
$\by =\bG(x,\by)$ has a solution $\bT(x)\in\bbR[[x]]_0^3$, and the solution is
$>\bzero$
.
The inverse of $I - J_\bG(0,0,0,0)$ is a non-negative matrix:
$$
\big(I - J_\bG(0,0,0,0)\big)^{-1}\ =\ 
\left[
\begin{array}{c  c  c}
1 & 0 & 0\\
0 &1 &0\\
1 &1 &1
\end{array}
\right].  
$$
Thus 
$$
\widehat{\bG}(x,\by)\ =\ 
\left[
\begin{array}{c }
x \,+\, 3 x y_1 y_2^2 \,+\, 3 x y_1^2 y_2\\
x \,+\, x y_3^2\\
2 x \,+\, 3 x y_1 y_2^2 \,+\, 3 x y_1^2 y_2 \,+\, xy_3^2 \,+\, y_1 \,+\, y_2
\end{array}
\right].  
$$

The spectral system $\bY = \widehat{\bGamma}(\bY)$ is
 \begin{eqnarray*}
 Y_1 &=& 1 \cup \big(1  \,+\,  Y_1  \,+\,  2\star Y_2\big) \cup \big(1  \,+\,  2\star Y_1  \,+\,  Y_2\big)\\
 Y_2 &=& 1 \,\cup\, \big(1  \,+\,   2\star Y_3\big)\\
 Y_3 &=&  1\,\cup\, \big(1  \,+\,   Y_1   \,+\,  2\star Y_2\big) \,\cup\, 
 \big(1  \,+\,  2\star Y_1  \,+\,  Y_2\big) \,\cup\, \big(1  \,+\,    2\star Y_3\big)
 \cup Y_1 \cup Y_2
 \end{eqnarray*}
\end{example}
%
 
%
\section{General Systems} \label{Gen Sys}
Recall that
\begin{eqnarray*}
   \dom_0[x] &=& \big\{A(x) \in\bbR[x] : A(0) = 0, A(x)  \unrhd 0\big\}\\
   \dom_0[x,\by] &=& \big\{\bG(x,\by)  \in\bbR[x]^k : \bG(x,\by)\unrhd \bzero, 
   \bG(0,\bzero)= \bzero,  J_\bG(0,\bzero)= \bzero\big\} .
\end{eqnarray*}
The systems $\by = \bG(x,\by)$ considered so far are power-series systems. 
However these are not adequate to capture the scope of the popular constructions 
such as $\MSet$ (multiset) and $\Cycle$ used in combinatorial specifications---in particular
one needs $\MSet$ in the study of monadic second--order classes in $\S$\ref{MSO Sec}. 

If $\cA$ and $\cB$ are two combinatorial classes with the same generating function, 
that is, $A(x)=B(x)$, then $\MSet(\cA)$ and $\MSet(\cB)$ have the same generating 
function; 
likewise for the construction $\Cycle$. Such constructions are called \textit{admissible} in
Flajolet and Sedgewick \cite{FlSe2009}.
In the case of $\MSet$, the generating function for $\MSet(\cA)$ is
$$
\exp\Big(\sum_{m\ge 1} A(x^m)/m\Big) - 1.
$$
Ordinary generating functions $A(x)$ have integer coefficients;
the operator $\MSet$ is extended to all $A(x)\in\dom_0[x]$
 by the same expression:
$$
\MSet(A(x))\ :=  \  \exp\Big(\sum_{m\ge 1} A(x^m)/m\Big) - 1.
$$
This operator cannot be expressed by a power series in $A(x)$, so specifications using 
$\MSet$ do not, in general, lead to elementary systems.

The operations and constructions/operators considered here are 
(see \cite{FlSe2009} or \cite{BBY2006}):
\begin{thlist}
\item
the constant $\bullet$ (a single node) corresponds to the polynomial $x$ in 
generating functions

\item
the construction $union$ (disjoint union) corresponds to the operation of $+$ (addition)  
for generating functions

\item
the construction $+$ (disjoint sum) corresponds to the operation $\times$ (product)  for 
generating functions

\item
the construction/operator $\Seq$ (sequence)

\item
the construction/operator $\MSet$ (multiset)

\item
the construction/operator $\Cycle$ (cycle)

\item
the construction/operator $\DCycle$ (directed cycle)
\end{thlist}
Items (d)--(g) are called the \emph{standard} constructions.
A standard construction $\Theta$ can be restricted to a set of positive integers $J$, giving the 
construction $\Theta_J$, the meaning of which is that $\Theta_J(\cA)$ consists of all objects 
that one can construct by applying $\Theta$ to only $J$-many objects from $\cA$ (repeats allowed).
Thus $\MSet_{\sf Even}(\cA)$ gives all multisets consisting of an even number of objects from $\cA$.  
The operators $J\star $, for $J \subseteq \bbN$, are precisely the operators $\MSet_J$, so 
the star operation ($\star$) is included in the above list.

\begin{definition}
Let $\cC$ be the collection of combinatorial classes. A construction
$${\Theta} :\cC^k \rightarrow \cC^m$$
is admissible iff:
\begin{quote}
whenever two $k$-tuples of combinatorial classes $\bcA$ and $\bcB$  
have the 
 same $k$-tuples of generating functions $\bA(x)$ and $\bB(x)$ then the $m$-tuple of
 combinatorial classes ${\bTheta}(\bcA)$ and 
 ${\bTheta}(\bcB)$
 also have the same $n$-tuples of generating functions.
 \end{quote}
 The operator from $\dom_0[x]^k$ to $\dom_0[x]^m$ induced by such a construction 
 is also
 designated by $\bTheta$.
 \end{definition}
 
 A variant of this definition is needed for the study of spectra of solutions to systems of
 equations.
 
\begin{definition}
An operator
$$\bTheta :\dom_0[x]^k \rightarrow \dom_0[x]^m,$$
is {\em spectrally admissible} provided:
\begin{quote}
whenever two $k$-tuples $\bA(x)$ and $\bB(x)$ from $\dom_0[x]^k$ 
have the same spectra, that is, $\bA = \bB$, then  
 $\bTheta(\bA(x))$ 
 and 
 $\bTheta\big(\bB(x)\big)$
 also have the 
 same  spectra, that is,
  $\Spec\Big(\bTheta\big(\bA(x)\big)\Big) = 
\Spec\Big(\bTheta\big(\bB(x)\big)\Big).$
 \end{quote}
 The operator from $\Su(\bbN)^k$ to $\Su(\bbN)^m$, where  $\Su(\bbN)$  is the set of subsets of $\bbN$, induced by a spectrally admissible
operator $\bTheta$ is 
 designated by $\bGamma_\bTheta$.
 \end{definition} 
 
 \begin{lemma}
Each $\bG(x,\by)\in\dom_{J0}[x,\by]^k$ defines an operator 
  on $\dom_0[x]^k$  that is both admissible and spectrally admissible. 
  Such operators are called {\em elementary operators}.
  As a spectrally admissible operator, $\bG(x,\by)$
induces a set-operator $($on $\Su(\bbN)^k$, the set of $k$-tuples of
subsets of $\bbN$$)$, namely
 $$
   \bGamma : \bA\mapsto \bigcup_{\bu\in\bbN^k} \bG_{i,\bu} + (\bu\star \bA).
 $$
 \end{lemma}
 %

\begin{definition}
Two spectrally admissible operators $\bTheta$ and $\bTheta'$ on $\dom_0[x]^k$ 
are {\em spectrally equivalent} if 
they give the same set-operator, that is, for all $\bA(x)\in\dom_0[x]^k$,
$$
\Spec\big(\bTheta\big(\bA(x)\big) \big)\   =  \  \Spec\big(\bTheta'\big(\bA(x)\big)\big).
$$
\end{definition}

  The standard admissible operators (and their restrictions) map $\dom_0[x]$ to itself, 
 hence $k=m=1$ in such cases. However the elementary operators require that one 
 take arbitrary $k\ge 1$ into consideration.
 
 In addition to the (restrictions of the) standard constructions $\Theta$ being admissible, they are
 spectrally admissible. 
A simplifying feature of working with spectrally admissible operators is that they can often be
better understood by replacing them with equivalent elementary operators.

\begin{theorem}[Systems based on Spectrally Admissible Operators] \label{gen spec thm}
\begin{thlist}
\item
Elementary operators $\bG(x,\by)$ and restrictions $\Theta_J$ of the standard operators $\Theta$ are 
spectrally admissible.

\item
The restriction $\Theta_J(y)$ of a standard operator $\Theta$ is spectrally equivalent 
to the elementary operator
$\displaystyle \sum_{j\in J} y^j$, and $\Spec\big(\Theta_J\big(A(x)\big)\big) = J\star A$.

\item
The sum $\Theta_1 + \Theta_2$, product $\Theta_1 \cdot \Theta_2$ and composition 
$\Theta_1 \circ \Theta_2$ of spectrally admissible 
operators is spectrally admissible.

\item
Any combination of elementary operators and restrictions of standard operators --- 
using the operations
of sum, product and composition --- yields an operator that is spectrally admissible
and spectrally equivalent to an elementary operator.

\item
If $\bTheta(\by)$ is spectrally equivalent to $\bTheta'(\by)$ then 
$$
\Spec\big(\bTheta^{(\infty)}(\bempty)\big)\   =  \  \Spec\big(\bTheta'^{(\infty)}(\bempty)\big).
$$

\item \label{ops}
Let $\by = \bTheta(\by)$ be a system with solution $\bT(x)\in\dom_0[x]^k$,
 where the operators 
$\Theta_i$ are combinations as described in item {\em (d)}.  
By {\em (d)},  $\bTheta$ is spectrally equivalent to  an elementary operator 
   $\bG(x,\by)$. 
   Let $\bU(x)$ be the unique solution to  $\by = \bG(x,\by)$ guaranteed by 
   Theorem \ref{Thm PSeries}.
Then $\bT = \bU$. 

Thus periodicity properties for the $T_i(x)$ can be deduced by applying
Theorem \ref{Thm PSeries} to  $\by = \bG(x,\by)$. 
\end{thlist}
\end{theorem}

\begin{proof}
Items (a) through (e) are straightforward.
For item (f), the operators $\Theta_i(\bY)$ are spectrally equivalent to an
elementary operator by (d).
From the spectral equivalence of the operators $\bTheta(\by)$ and $\bG(x,\by)$
and the fact that $\bT(x)$ is a solution of $ \by = \bTheta(\by)$, one has
$$
\bT\   =  \  \Spec\big(\bTheta\big(\bT(x)\big)\big)\   =  
\  \Spec\big(\bG\big(x,\bT(x)\big)\big)\   =  \  \bGamma(\bT),
$$
where $\bgam$ is the set operator corresponding to $\bG(x,\by)$.
So $\bT$ is a solution of 
$
\bY = {\bGamma}\big(\bY\big).
$
Now $\bU(x) = \bG(\bU(x))$ implies that $\bU$ is also a solution of
$
\bY = {\bGamma}\big(\bY\big).
$
Theorem \ref{ThmSetEqn} says that  the elementary system $\bY = \bGamma(\bY)$ has a 
unique solution, so
 $\bU =  \bT$.
 Consequently the periodicity properties of $\bT$ are 
those of $\bU$, and thus Theorem \ref{Thm PSeries} can be used to
analyze $\bT$.
\end{proof}

The next example illustrates the methods for determining the periodicity parameters
for the power series solution $\bT(x)$ of a general equational system $\by = \bTheta(\by)$
using a specification of  `structured' trees,\footnote
{ This additional structure on the tree can be viewed as a way of embedding a tree in 
3-space, so that a node that covers a cycle of nodes `looks' rather like a chandelier; 
perhaps one would prefer to consider the structure to be maintained by the 
legendary substance called quintessence that fixed the stars in the ancient heavens ---  it was invisible, weightless, etc.}
 where (some or all of) the nodes immediately below a node can be given a structure,
such as a cycle or a sequence.
\begin{example}
Let $\cT$ be the class of two-colored $($red,blue$)$ `structured' trees which satisfies 
the following conditions:
\begin{thlist}
\item
A red node must have a cycle consisting of a positive even number of red nodes,
or 6 blue nodes, at least 3 of the blue nodes being leaves,
 immediately below it;
\item
A blue node that is not a leaf has a multiset consisting 
of a prime number of 
red nodes immediately below it, plus a sequence of  blue nodes 
whose number is  congruent to 4 mod 6.
\end{thlist}
Letting $\cR$ be the members of $\cT$ with a red root, and $\cB$ those with a blue root, one has
the specification
\begin{eqnarray*}
\cR&=&\frac{\bullet_R}{ \Cycle_\EvenSquares (\cR )}\ \cup 
\ \frac{\bullet_R}{3 \bullet_B\ + \  \MSet_3(\cB)}\\
\cB&=&\{\bullet_B\}\ \cup\ \frac{\bullet_B}
{ \MSet_{\Primes}(\cR) \ +\ \Seq_{4 + 6\cdot \bbN}( \cB)}\\
\cT&=&\cR \cup \cB.
\end{eqnarray*}
The associated spectral system is
\begin{eqnarray*}
Y_1&=&(1 + \EvenSquares \star Y_1)\  \cup\  (4+ 3\star Y_2)\\
Y_2&=&1 \cup \Big(1+ \Primes\star Y_1 + (4 + 6\cdot \bbN)\star Y_2\Big)\\
Y_3&=&Y_1 \cup Y_2
\end{eqnarray*}
with solution $(Y_1,Y_2,Y_3) = (R,B,T)$.
This is not an elementary system $($because of the linear terms in the right side of 
the third equation$)$,
but nonetheless the solution is unique.
Note that $\{1,2\}$ is a strong component of the dependency digraph.

To determine the periodicity parameters for $\cT$ it suffices to determine 
them for $\cB$ and $\cR$ 
and apply Proposition \ref{Karen's Table}, since $\fm(T) = \fm(R\cup B)$ and
$\fq(T) = \fq(R\cup B)$. 
The  first two equations form
an elementary system, and one has: 
\begin{eqnarray*}
\bGamma(Y_1,Y_2)&=&
\left(
\begin{array}{c c}
(1 + \EvenSquares \star Y_1)\  \cup\  (4+ 3\star Y_2)\\
1 \cup \Big(1+ \Primes \star Y_1 + (4 + 6\cdot \bbN)\star Y_2\Big)
\end{array}
\right)\\
\bGamma(\O,\O)&=&
\left(
\begin{array}{c c}
\O\\
1
\end{array}
\right)\\
\bGamma^{(2)}(\O,\O)&=&
\left(
\begin{array}{c c}
7 \\
1
\end{array}
\right).
\end{eqnarray*}
Thus $\bfm := (\fm_1,\fm_2) = (7,1)$.

Writing 
$$
\Gamma_i(Y_1,Y_2) = \bigcup_{\bu\in\bbN^2} G_{i,\bu} + (u_1\star Y_1 + u_2\star Y_2)
$$
one has
\begin{eqnarray*}
G_{1,\bu}&=&
\begin{cases}
1&\text{if } u_1  \in \EvenSquares \text{ and }u_2=0\\
4&\text{if }u_1=0 \text{ and } u_2 = 3\\
\O&\text{otherwise}
\end{cases}\\
G_{2,\bu}&=&
\begin{cases}
1&\text{if }(u_1=u_2=0) \text{ or } (u_1 \in \Primes \text{ and }  u_2 \equiv 4 \text{ mod } 6)  \\
\O&\text{otherwise}.
\end{cases}
\end{eqnarray*}
Now $\bu\star\bfm = 7u_1 + u_2$, so
\begin{eqnarray*}
G_{1,\bu} + \bu\star\bfm - \fm_1&=&
\begin{cases}
7u_1 + u_2 - 6 &\text{if }  (u_1  \in \EvenSquares \text{ and }u_2=0) \\
7u_1 + u_2 - 3 &\text{if } u_1=0 \text{ and } u_2 = 3\\
\O&\text{otherwise}
\end{cases}\\
G_{2,\bu}+ \bu\star\bfm - \fm_2&=&
\begin{cases}
 7u_1+u_2 &\text{if }(u_1=u_2=0) \text{ or } (u_1 \in \Primes 
 \text{ and }  u_2 \equiv 4 \text{ mod } 6)\\
\O&\text{otherwise}.
\end{cases}
\end{eqnarray*}
From this one has
\begin{eqnarray*}
\gcd\bigcup_\bu\big(G_{1,\bu} + \bu\star\bfm - \fm_1\big)&=&
\gcd\{7u_1  - 6:  u_1  \in \EvenSquares \} \ =\ 2\\
\gcd\bigcup_\bu\big(G_{2,\bu}+ \bu\star\bfm - \fm_2\big)&=&
 \gcd\{7u_1 +u_2 : u_1 \in \Primes \text{ and }  u_2 \equiv 4 \text{ mod } 6\}\ =\ 1\\
\end{eqnarray*}
Since the two equation system is irreducible, that is, $i\rightarrow^+j$ for all vertices $i,j$,
one has
\begin{eqnarray*}
\fq_1&=&\fq_2\ =\  \gcd \bigcup_\bu \bigcup_{i=1}^2 G_{i,\bu} + \bu\star\bfm - \fm_i \\
&=&\gcd(2,1)\ =\ 1.
\end{eqnarray*}
Using Proposition \ref{Karen's Table},
the above calculations give 
\begin{eqnarray*}
\fm_3 &=&\min(\fm_1,\fm_2 )\ =\  \min(7,1)\ =\ 1\\ 
\fq_3 &=& \gcd(\fq_1,\fq_2,\fm_1 - \fm_2)\  =\  \gcd(1,1,6)\  =\  1.
\end{eqnarray*}
In summary, $(\fm(R),\fm(B), \fm(T)) = (7,1,1)$ and 
$(\fq(R),\fq(B), \fq(T)) = (1,1,1)$.
\end{example}

%
\section{Monadic Second Order Classes} \label{MSO Sec}
At present there are two major approaches to describing broad collections of
combinatorial structures: 
(1) combinatorialists (see, for example, \cite{FlSe2009}) prefer to look at \textit{specifications} 
that are based on constructions like \textit{sequences},  \textit{cycles} and \textit{multisets},
 whereas 
(2) logicians prefer to look at classes that are defined by \textit{sentences in a formal logic}.

When working with relational structures like graphs and trees, logicians have found it 
worthwhile to strengthen first-order logic to {\em monadic second-order logic}
(MSO logic).\footnote
	{This is just first-order  logic  augmented with unary predicates $U$ as variables  --- 
	 this means that one can quantify over subsets as well as individual elements, 
	 and say that an element belongs to a subset. 
	The fact that the  $U$ are predicates and not domain  elements make the logic 
	\textit{second-order}, and the fact that these predicates have only one argument (e.g., 
	$U(x)$) makes the  logic \textit{monadic}.}
The  primary reason for the interest in MSO logic is the powerful connection between 
Ehrenfeucht-Fra\"{i}ss\'{e} games and sentences of a given quanifier rank.\footnote
	{The connection with Ehrenfeucht-Fra\"{i}ss\'{e} games fails if one has quantification over 
	more general relations, like binary relations.} 
These games, although very combinatorial in 
nature, are not widely used in the combinatorics community.

\subsection{Regular Languages}
A set $\cL$ of words over an $m$-letter alphabet is a {\em regular language} if it is precisely the
set of words accepted by some finite state deterministic automaton. A word is accepted by 
such an automaton if, starting at state 0, one can follow a path to a final state with the successive
edges of the path spelling out the word. Let the states of the automaton be $S_0,\ldots, S_k$, and for 
each state $S_i$ let $\cL_i$ be the set of words traversed when going from vertex $i$ to a final state vertex. 
Then one sees that $\cL_i$ is the union of the classes $a_{ij}\cL_j$ where $i\rightarrow j$ 
is an edge in the automaton labeled by the letter $a_{ij}$ from the alphabet. 
This leads to equations 
of a particularly simple form for the generating functions and the spectra, namely 
for $1 \le i \le k$,
\begin{eqnarray*}
L_i(x)&=&x\cdot\big(c_i + \sum_{i\rightarrow j} L_j(x) \big)\\
L_i&=&A_i \cup \big(1 + \bigcup_{i\rightarrow j} L_j\big) .
\end{eqnarray*}
One of the first big successes for MSO was B\"uchi's Theorem connecting the regular languages
studied by computer scientists with classes of colored digraphs defined by MSO sentences. 
To see how this connection is made,
simply note that a word on $m$ letters corresponds to an $m$-colored linear digraph 
 $(D,\rightarrow,C_1,\ldots,C_m)$,
  and thus a language on an $m$-letter alphabet can be thought of as a class of $m$-colored 
  linear digraphs.
 
 \begin{theorem}[B\"uchi \cite{Buchi1960}, 1960]  
 MSO  classes of colored linear digraphs  are precisely the  regular  languages.
 \end{theorem}
 
 The theory of the generating functions for MSO classes of colored linear 
 digraphs was worked out, in 
 the context of regular languages, by Berstel \cite{Berstel1971}, 1971
 (his results were soon augmented by Soittola \cite{Soittola1976}, 1976). 
 Given a regular language $\cR$, one can partition it into classes $\cR_i$
such that the generating functions $R_i(x)$ satisfy
a system  of  linear equations
$\by = x(C + M\cdot \by)$, where $C$ is a 0,1-column matrix,
and $M$ is a 0,1-square matrix.
The equations are easily read off a finite state deterministic automata that 
accepts the language; one writes down a
 system of equations for the paths in the automata, similar
 to the situation in Example \ref{paths}. The equations have a particularly 
 simple linear form---the spectra $R_i$ are eventually periodic, and by Cramer's rule,
the generating functions $R_i(x)$ are rational functions; also they 
are given by
$
\bR(x) = x\cdot \big(I - xM\big)^{-1}\cdot C.
$
 Berstel showed that
  each $R_i(x)$ decomposes into a finite number of  $R_{ij}(x)$, 
  each $R_{ij}$ being either finite or eventually an arithmetical progression. For those
  which are not finite
 there are polynomials $P_{ijk}(n)$ and complex numbers
 $\beta_{ijk} = \beta_{ij}\cdot \omega_{ij}^k$, with $\beta_{ij}$ a positive real and
 $\omega_{ij}$ a root of unity, such that,
on the set $R_{ij}$, one has the coefficients $r_{ij}(n)$
having an exact  polynomial-exponential form, and polynomial-exponential 
asymptotics, given by (see \cite{FlSe2009}, p. 302):
\begin{eqnarray*}\label{Berstel}
r_{ij}(n)& = &  \sum_k P_{ijk}(n) \beta_{ijk}^n\quad\text{on the set } R_{ij} \\
&\sim&P_{ij0}(n) \beta_{ij}^n\quad\text{on the set } R_{ij} .
\end{eqnarray*}

This study of the generating functions for MSO classes of colored 
linear  digraphs provides {\em the Berstel Paradigm}, a successful
analysis that one would like to see paralleled in the study of all 
MSO classes of colored trees. For example, can one show that the generating 
functions $T(x)$ of such classes decompose into a 
 polynomial and finitely many ``nice'' functions $T_i(x)$, with each spectrum
 $T_i$ being an arithmetical progression?

\subsection{Trees and Forests}

When speaking of structures, in particular the models of a sentence $\varphi$, 
it will be understood that only finite structures are being considered.

A \textit{tree} $\bT = (T,<)$ is a poset such that:  
(i) there is a unique maximal element $\rt(\bT)$ 
called the \textit{root} of the tree, and
(ii) every interval $[a,\rt(\bT)]$ is linear. 
A \textit{forest} $\bF = (F,<)$ is a poset whose components are trees.

A forest $\bF$ is determined (up to isomorphism) by the number 
of each (isomorphism type of) tree appearing in it, thus by its
counting function $\nu_\bF : \TREES \rightarrow \bbN$.

One can combine two forests $\bF_1$ and $\bF_2$ into a single forest 
$\bF_1 + \bF_2$ which is determined 
up to isomorphism by  $\nu_{\bF} = \nu_{\bF_1} + \nu_{\bF_2}$. 
Extend this operation to classes $\cF$ of forests
by 
$\cF_1 + \cF_2 = \{\bF_1 + \bF_2 : \bF_i \in \cF_i\}$. 
The \textit{ideal class} $\cO$  of forests is introduced with
the properties $\cO\cup \cF =\cO + \cF = \cF$ (it is introduced solely as a notational device
to smooth out the presentation).

Define the operation $\star$ between non-empty subsets $A$ of $\bbN$ and 
non-empty classes $\cF$ of forests by 
\begin{eqnarray*}
n\star \cF &=& 
\begin{cases} \cO &\text{if }n=0\\
\underbrace{\cF + \cdots + \cF}_{n-\text{fold}}&\text{if } n\ge 1
\end{cases}\\
A \star \cF &=& 
\bigcup_{a\in A}a \star \cF.
\end{eqnarray*}

\subsection{Compton's Specification of MSO Classes of Trees}

$\bF_1\equiv_q^\MSO\bF_2$  means that $\bF_1$ and $\bF_2$ satisfy the same 
MSO sentences of quantifier rank $q$.
$\equiv_q^\MSO$ is an equivalence relation on $\cF$ of finite index. In the following, 
when given a MSO class $\cF$ of forests, it will be assumed that  $q$ has been chosen  
large enough so (i) $\cF$ is definable by a MSO sentence of quantifier depth $q$,
 and (ii) that there 
are MSO sentences of quantifier depth $q$ to express ``is a tree'', ``is a forest''. 
Then $\TREES$ is a union of $\equiv_q^\MSO$ classes of forests, say 
$\TREES = \cT_1\cup\cdots\cup\cT_r$.
 If $\cT_i$ has a 1-element tree in it then no other tree is in $\cT_i$.
Assume that there are $m$ colors, and let $\bullet_i$ denote the 1-element tree of color $i$, 
and assume $\cT_i = \{\bullet_i\}$ for $1\le i \le m$. These are the only $\cT_i$ with a 
one-element member, and all trees in any given $\cT_i$ have the same root-color, say $i'$.

Given a tree $\bT$ with more than one element, let $\partial \bT$ be the forest that 
results from removing the root $\rt(\bT)$ from $\bT$; and given any forest $\bF$, 
let $\bullet_i/\bF$ be the tree that results by adding a root of color $i$ to the forest. 
The operation $\partial$ is extended in the obvious manner to $\partial \cT$ for
any non-empty class $\cT$ of trees that does not have a one-element tree in it; 
and the operation of adding a root of color $i$ to a forest is extended to $\bullet_i/\cF$ 
for any non-empty class $\cF$ of forests. 
 
 %
 \begin{lemma}\label{nf tools}
 Let $q$ be a positive integer.
 \begin{thlist}
\item
The operations of disjoint union and $\bullet_i\big/$  preserve $\equiv_q^{\sf MSO}$,
that is, 
\subitem 
$\displaystyle T_i\equiv_q^{\sf MSO} T_i'\  
\Rightarrow \ 
\sum_i T_i \equiv_q^{\sf MSO} \sum_i T_i'$, and 
\subitem 
$\displaystyle F \equiv_q^{\sf MSO} F' \ 
\Rightarrow\  
\bullet_i/F \equiv_q^{\sf MSO} \bullet_i/F'$, for $1 \le i \le m$.
\item
There is a constant $C_q$ such that for all trees $T$ and all $n\ge C_q$ 
one has 
$n\star T \equiv_q^{\sf MSO} C_q\star T$.
\item
There is a decision procedure to determine if $F_1 \equiv_q^{\sf MSO} F_2$.
\end{thlist}
\end{lemma}
\begin{proof}
One can find a discussion of the first item of (a), as well as item (b),  
in \cite{Burris2001}, based on E-F games. 
Use E-F games for  (c) as well.  (a)--(c) are basic tools of Gurevich and Shelah
(\cite{GuSh2003}, 2003).
\end{proof}

 The next lemma gives the crucial structure result for MSO classes of forests.
 
 \begin{lemma} \label{normal form}
 Let $\cF$ be a MSO class of forests defined by a sentence of quantifier rank $q$.
 Then there is a finite set $\bbS$ of $r$-tuples $S = (S_i)$ of cofinite or non-empty 
 finite subsets $S_i$ of $\bbN$
 such that 
 \begin{equation*}\label{expr F}
 \cF\ =\ \bigcup_{S\in\bbS} \sum_{\substack{j=1\\ S_j\neq\{0\}}}^r S_j \star \cT_j.
 \end{equation*}
 \end{lemma}
 \begin{proof}
 Let $C_q$ be as in Lemma \ref{nf tools}.
 A routine application of Ehrenfeucht-Fra\"{i}ss\'{e} games shows that
  for any two $r$-tuples $(n_i)$ and $(n_i')$ of non-negative integers with
  $n_i\ge C_q$ iff $n_i' \ge C_q$, one has every member
  of $\sum_i n_i\star \cT_i $ equivalent modulo $\equiv_q^\MSO$ to every member of 
$\sum_i n_i'\star \cT_i $. 

  Thus $\cF$ decomposes into a (disjoint) union of finitely many classes 
  $\sum_{i=1}^r S_i \star\cT_i$ where each $S_i$ is either a singleton $\{n_i\}$ with
  $1\le n_i \le m$ or the cofinite set $\{n : n\ge C_q\}$.
 \end{proof}
 
 \begin{lemma}\label{cut T MSO}
 For $m < i\le r$, the class of forests $ \partial \cT_i$ is definable by a MSO sentence 
 of quantifier rank $q$.
\end{lemma}
\begin{proof}
$\partial \cT_i$ is closed under $\equiv_q^\MSO$ since
Lemma \ref{nf tools} shows $\bF_1 \equiv_q^\MSO \bF_2$ implies 
$\bullet_i\big/\bF_1 \equiv_q^\MSO \bullet_i\big/\bF_2$, thus 
$\bF_1 \equiv_q^\MSO \bF_2$ and $\bF_1 \in \partial\cT_i$ imply $\bF_2 \in \partial\cT_i$. 
\end{proof}

\begin{theorem}[Compton, see \cite{Woods1997}] \label{Thm Kevin}
Let $\cT$ be a class of $m$-colored trees defined by a MSO sentence of quantifier depth $q$. 
Then:
\begin{thlist}
\item
$\cT$ is a union  of some of the $\cT_i$, and
\item
the $\cT_i$ satisfy a system of equations  
$$
\Sigma_q :\ 
\left\{
\begin{array}{r c l}
   \cT_1  & = &    \Phi_{1}\big(\cT_{1},\ldots,\cT_{r}\big)\\
               &\vdots&\\
   \cT_r  & = &    \Phi_{r}\big(\cT_{1},\ldots,\cT_{r}\big),
\end{array}
\right.
$$
\end{thlist}
where $\Phi_i\big(\cT_1,\ldots,\cT_r\big)$ is  $\{\bullet_i\}$ for $1\le i\le m$, 
and for $i>m$ it has the form 
\begin{equation}\label{Fnormal form}
  \bullet_{i'}\bigg/ \bigcup_{S \in \bbS_i} \sum_{j=1}^r S_j\star \cT_j
\end{equation}
with each $\bbS_i$ being a finite set of  $r$-tuples $S = (S_1,\ldots,S_r)$, 
with each $S_j$ a cofinite or non-empty finite  subset of $\mathbb N$.
\end{theorem}
\begin{proof}
(a) is obviously true. For (b) note that for $m<i\le r$, $\cT_i  = \bullet_i\big/\partial \cT_i$.
Lemma \ref{cut T MSO} 
says
$\partial \cT_i$  is definable by a MSO sentence of quantifier rank $q$. Then Lemma \ref{normal form}
shows that
$\partial \cT_i$ can be expressed in a particular form.
One only needs to attach the root $\bullet_{i'}$ to have
\eqref{Fnormal form}.
\end{proof}

Applying $\Spec$ to  $\Sigma_q$  gives a  system of set-equations for 
the spectra of the classes $\cT_i$:

\begin{corollary}\label{spec eqns}
For $\cT$ as in Compton's Theorem, $\Spec(\cT)$ is a union of some of the $\Spec(\cT_i)$, and
\begin{eqnarray}\label{Kevin spec}  \displaystyle
   \Spec(\cT_i )& = & 
   \begin{cases}
  \{1\}&\text{for }1\le i\le m\\  
  \displaystyle
   \{1\} + 
   \bigcup_{S \in \bbS_i} \sum_{j=1}^r S_j\star \Spec(\cT_j )&\text{for }m<i\le r.
   \end{cases}  
\end{eqnarray}
\end{corollary}

\begin{remark} 
Compton \cite{Compton2009} described his equational specification for the minimal MSO 
classes of trees of quantifier depth $q$ to Alan Woods during a visit to Yale in 1986; at the time Woods was a PostDoc at Yale. Evidently Compton regarded such an equational specification for
trees as a straightforward generalization of the earlier work of B\"uchi, which
showed that regular 
languages were precisely the MSO  classes of $m$-colored linear trees.
\end{remark}

\subsection{The dependency digraph of  $\Sigma_q$}
The dependency digraph $D_\fq$ for $\Sigma_q$ is defined parallel to the definition for 
systems of set-equations.
 $D_q$ has vertices $1,\ldots,r $  and, referring to \eqref{Kevin spec}, directed edges given
 by $i \rightarrow j$ iff  there is a $S\in\bbS_i$ such
 that $S_j\neq\{0\}$. 
%
%
One defines a \textit{height function} on $D_q$ by setting $h(i)=0$ for $1\le i\le m$, and
then for $m<i\le r$ use the inductive definition $h(i) = 1+\max\{h(j) : i \rightarrow^+j, 
\text{ but not } j\rightarrow^+ i\}$.

\begin{corollary}
The spectrum of a MSO class $\cT$ of $m$-colored trees is eventually periodic.
\end{corollary}
\begin{proof}
It suffices to prove this result for the $\cT_i$ in view of Lemma \ref{ground case} (which 
guarantees that eventual periodicity is preserved by finite union). 
For $1\le i\le m$ this is trivial. So suppose $m<i\le r$, and
 note that whenever $j\rightarrow k$ one has $\Spec(\cT_j) \supseteq p_{jk} + \Spec(\cT_k)$
for some positive integer $p_{jk}$, by \eqref{Kevin spec}. Thus $i\rightarrow^+j$ implies the same
conclusion. If $[i]\neq \O$ then $i\rightarrow^+ i$, 
so $\Spec(\cT_i)\supseteq p + \Spec(\cT_i)$ for some $p\in \bbP$,
so $\Spec(\cT_i)$ is actually periodic. If $[i] = \O$ then one argues, by induction on the height $h(i)$, that 
$\Spec(\cT_i)$ is eventually periodic. The ground case, $h(i)=0$, holds precisely for $1\le i\le m$, and
in these cases $\Spec(\cT_i) = \{1\}$, an eventually periodic set.  Now suppose the result holds for
$h(i) \le n$. If $h(i)= n+1$ then $m< i\le r$, and one has
\begin{eqnarray*}
\Spec(\cT_i)&=&  \{1\} + \bigcup_{S \in \bbS_i} 
\sum_{j=1}^r S_j\star \Spec(\cT_j )
\end{eqnarray*}
For the  $j$ such that there is an $S$ with $S_j \neq \{0\}$ 
(there is at least one such $j$ since $i>m$)
 one has $i \rightarrow j$, so $h(j)< h(i)$, implying $\Spec(\cT_j)$
is eventually periodic (by the induction hypothesis). The $S_j$ are either cofinite or non-empty finite,
and therefore eventually periodic. Then Lemma \ref{ground case} shows 
$\Spec(\cT_i)$ is eventually periodic,  since being eventually periodic is preserved by finite unions,
(finite) sums, and $\star$, 
 with the additional information that those $\cT_i$ belonging to a strong component are 
actually periodic.
\end{proof}

\begin{corollary}
The spectrum of a MSO class $\cF$ of $m$-colored forests is eventually periodic.
\end{corollary}
\begin{proof}
Since $\bullet_1\big/\cF$ is a MSO class of trees one has $\{1\}+\Spec(\cF)$ eventually periodic,
hence so is $\Spec(\cF)$.
\end{proof}

\begin{theorem}[Gurevich and Shelah \cite{GuSh2003}, 2003] \label{Thm GuSh}
Let $\cU$ be a MSO class of $m$-colored unary functions.
 Then the spectrum $\Spec(\cU)$ is eventually periodic.
\end{theorem}
\begin{proof}
It suffices to show that  one can find an MSO class of $m$-colored 
forests with the same spectrum. 
Let $\cF$ be the class of $m$-colored forests defined as follows:
\begin{quote} \sf
 for each forest in the class there exists a subset $V$ of the forest, with exactly one element from 
 each tree in the forest, such that if one adds a directed edge from the root of each tree in the
 forest to the unique node of the tree in $V$, then
one has a digraph which satisfies a defining sentence of $\cU$.
\end{quote}
Clearly this condition can be expressed by a MSO sentence, so $\Spec(\cF)$ is eventually
periodic; hence so is $\Spec(\cU)$.
\end{proof}

Although the proof of the Gurevich and Shelah Theorem comes after considerable 
develoment of the theory of spectra defined by equations, actually what is needed
for this proof, beyond Compton's Theorem, is Lemma \ref{ground case}.
 This theorem is almost best 
possible for MSO classes --- for example, one cannot replace `unary function' with 
`digraph' or `graph' as one can easily find classes of such structures where the theorem 
fails to hold. The converse, that every eventually periodic set $S\subseteq \bbP$ can be realized
as the spectrum  of  a  MSO sentence for unary functions is easy to prove.

In a related direction one has the following :
 \begin{corollary}\label{defect}
 A MSO class of  graphs with bounded defect has an eventually periodic spectrum.
 \end{corollary}
 \begin{proof}
 A connected graph has defect $d$ if $d+1$ is the minimum number of edges that need to  be 
 removed in order to have an acyclic graph. Thus trees have defect = -1. A graph has defect
 $d$ if the maximum defect of its components is $d$. For graphs of defect at most $d$,
  introduce $d+2$ colors, one to mark a choice of a root in each component, and the others 
  to mark the endpoints of edges which, when removed, convert the graph into a forest. For
  an MSO class of  $m$-colored graphs of defect at at most $d$, carrying out this additional 
  coloring in all possible ways gives an MSO class of $m+d+2$ colored graphs. Then removing
  the marked edges from each graph converts this into a MSO class of 
  $m+d+2$ colored forests with the same spectrum.
 \end{proof}
 
 This can be easily generalized further to MSO classes of digraphs with bounded defect, giving a slight
 generalization of the Gurevich-Shelah result (since trees have defect $-1$, unary functions have defect $0$). 
 These examples suffice to indicate the power
 of knowing that monadic second-order classes of trees have eventually periodic
 spectra. The method of showing that MSO spectra are eventually periodic
 by reducing them to trees has been successfully pursued by Fischer and Makowsky in \cite{FiMa2004} (2004), 
 where they prove that an MSO class that is contained in a 
 class of bounded patch-width has an eventually periodic spectrum.  In the same year Shelah \cite{Shelah2004} 
 proved that MSO classes having a certain recursive constructibility property had eventually periodic spectra, 
 and in 2007  Doron and Shelah \cite{DoSh2007} showed that the bounded patch-width result was a consequence 
 of the constructibility property.
 
\subsection{Effective Tree Procedures}
What follows is a program, given $q$, to effectively find a value for $C_q$ and 
representatives of the $\equiv_q^{\sf MSO}$ classes 
of  \TREES\ and of \FORESTS, with applications to the decidability results of 
Gurevich and Shelah (\cite{GuSh2003}, 2003), 
and an effective procedure to construct Compton's system of equations for trees. 
The particular classes of trees constructed in the WHILE loop of this program 
are similar to the classes $\cT_k^m$ used in 1990 by Compton and Henson to prove lower bounds
on computational complexity (see \cite{CoHe1990}, p.~38).

\renewcommand{\arraystretch}{1.2}
\begin{tabular}{l l}
Program Steps & Comments\\
\hline
{\sf FindReps} := PROC$(q)$&$q$ is the quantifier depth\\
$\cF_0 :=\O$ &  Initialize collection of forests\\
$\cT_{i,0} := \{\bullet_i\},\ 1\le i \le m$ &\begin{tabular}{l}Initialize collection of trees\\
\quad with root color $i$\end{tabular}\\
$\cT_0 := \{\bullet_1,\ldots,\bullet_m\}$&Initialize collection of trees\\
$f(0) := 0$ & cardinality of $\cF_0/\equiv_q^{\sf MSO}$\\
$t_{i}(0) := 1,\ 1\le i \le m$ & cardinality of $\cT_{i,0}/\equiv_q^{\sf MSO}$\\
$t(0) := m$ & cardinality of $\cT_0/\equiv_q^{\sf MSO}$\\
$d(0) := 1$&initialize $d(n)$\\
$n := 0$ & initialize $n$
\end{tabular}

\begin{tabular}{l l}
WHILE $f(n) > f(n-1)$  OR $d(n)>0$ DO\\
$n := n+1$ & augment the value of $n$\\
\quad $\displaystyle\cF_{n} := \Big\{ \sum_{T \in \cT_{n-1}} m_T \star T :  m_T \le n\Big\}$& 
\begin{tabular}{l}make forests using at most $n$ copies \\\quad of each tree in $\cT_{n-1}$\end{tabular}\\
\quad $\displaystyle \cT_{i,n} := \{\bullet_i\} \ \cup\ 
\big\{ \bullet_i \big/ F : F \in \cF_{n}\big\}$&add root of color $i$ to forests in $\cF_{n}$\\
\quad $\cT_{n} := \cT_{1,n}\cup\cdots\cup \cT_{m,n}$&$\cT_{n}$ has all trees created so far\\
\quad $t_i(n) := \big| \cT_{i,n}/\equiv_q^{\sf MSO}\big|$&\# of $\equiv_q^{\sf MSO}$ classes represented by $\cT_{i,n}$\\
\quad $t(n) := \big| \cT_{n}/\equiv_q^{\sf MSO}\big|$&\# of $\equiv_q^{\sf MSO}$ classes represented by $\cT_{n}$\\
\quad $f(n) := \big| \cF_{n}/\equiv_q^{\sf MSO}\big|$&\# of $\equiv_q^{\sf MSO}$ classes represented by $\cF_{n}$\\
\quad $d(n) := \Big| \big\{ T \in \cT_{n} : (n-1)\star T\, \nequiv_q\ n\star T \big\}\Big|$&
\begin{tabular}{l}
\# of failures of 
$(n-1)\star T \equiv_q^{\sf MSO} n\star T$\\
\quad for $T\in \cT_n$
\end{tabular}\\
 END WHILE
 \end{tabular}

\begin{tabular}{l l}
Define $C_q := n-1$\\
Choose a maximal set $REP_{\TREES} :=  \{T_1,\ldots,T_k\}$\\
\quad of $\equiv_q^{\sf MSO}$ distinct trees $T_i$ from $\cT_{C_q}$.\\
Choose a maximal set $REP_{\FORESTS} :=  \{F_1,\ldots,F_\ell\}$\\
\quad  of $\equiv_q^{\sf MSO}$ distinct forests $F_j$ from $\cF_{C_q}$.\\
 RETURN $(N, REP_\TREES,REP_\FORESTS)$\\
 END PROC\\
 \hline
\end{tabular}

\begin{theorem}
The procedure
 {\sf FindReps}$(q)$ halts for all $q\in \bbN$, giving an effective procedure 
 to find  
 a set $REP_\TREES$ 
 of representatives for the $\equiv_q^{\sf MSO}$ equivalence classes of (finite) 
trees, a set $REP_\FORESTS$ of representatives for the $\equiv_q^{\sf MSO}$ equivalence 
 classes of (finite)forests,
 and a number $N$ such  
that for any tree $T$ and $n\ge N$ one has $n\star T \equiv_q^{\sf MSO} N\star T$.
 \end{theorem}
 
 \begin{proof}
 The classes $\cF_n$ and $\cT_n$ are non-decreasing,  every (finite) forest is in some $\cF_n$, and
 every (finite) tree is in some $\cT_n$. 
 
 Based on comments in the introduction, 
let $P_q$ be the (finite) number of $\equiv_q^{\sf MSO}$ classes of finite forests, and let $C_q$ be such that 
$n\star T \equiv_q^{\sf MSO} C_q\star T$, for any tree $T$.
 
Since $f(n)$ is non-decreasing and $\le P_q$, there is an $R_q$ such that for $n\ge R_q$ one has $f(n) = f(R_q)$.

For $n> C_q$ one has $d(n) = 0$.
 
 So for $n > \max(C_q,R_q)$, the WHILE condition must fail to hold. Thus the looping process  in the procedure
 {\sf FindReps} halts for some $n\le \max(C_q,R_q)$.

 Since the number $N$ returned by the procedure is such that $f(N) = f(N-1)$ and $d(N) = 0$,
 every forest in $\cF_{N+1}$ is $\equiv_q^{\sf MSO}$
 to one in $\cF_N$, so $f(N+1) = f(N)$. Then
 $t(N+1) = t(N)$ and $d(N+1)=0$. 
 
By induction one has 
 $$
\Big(f(N) = f(N-1)\ \wedge\ d(N) = 0\Big)\ \Rightarrow \ (\forall n\ge N) \Big(f(n) = f(N)\ \wedge\ t(n) = t(N)\ \wedge\ d(N)=0\Big).
$$
 
 Consequently $\cF_N$ has representatives for all $\equiv_q^{\sf MSO}$ equivalence classes of forests, 
 and $\cT_N$ has representatives for all $\equiv_q^{\sf MSO}$ equivalence classes of trees, and 
 $N$ has the desired property of functioning as a value for $C_q$.
 
The procedures for constructing the classes $\cF_n$, $\cT_{i,n}$, and $\cT_n$ are effective, as
are the calculations of the functions $f(n)$, $t_i(n)$, $t(n)$ and $d(n)$. 
\end{proof}

%
Further Conclusions:
\begin{thlist}

\item
The trees in $\cT_n$ are all of height $\le n$,  $t_1(n) = \cdots = t_m(n)$, 
and $t(n)  = t_1(n) + \cdots + t_m(n) = m\cdot t_1(n)$.

\item
One can effectively find MSO sentences $\varphi_i$, $1\le i \le k$, such that $\varphi_i$ defines 
$ [T_i]_q$, the  $\equiv_q^{\sf MSO}$ equivalence class of trees of with the representative $T_i$ in it.

(Just start enumerating the sentences $\varphi$ and
 test each one in turn to see if  $(\exists i)(\forall j) \big(T_j\models \varphi\ \Leftrightarrow i = j\big)$. 
 If so then $i$ is unique; if no sentence
 had been previously found that defined $[T_i]_q$, then let $\varphi_i := \varphi$.)
 
 \item
 Likewise for $1 \le j \le \ell$ one can effectively find $\psi_j$ defining $[F_j]_q$, the  $\equiv_q^{\sf MSO}$ equivalence class of 
 forests with $F_j$ in it.
 
\item (Gurevich and Shelah, \cite{GuSh2003} 2003)
The MSO theory of FORESTS  is decidable. (Given $\psi$, it will be true of all forests
 iff it is true of each $F_j$ in $REP_\FORESTS$.)
 
\item (Gurevich and Shelah, \cite{GuSh2003} 2003)
Finite satisfiability for the MSO theory of one $m$-colored (finite) unary function is decidable. 
(This can be proved directly, by interpretation into FORESTS.)

\item
One can effectively find the Compton Equations $\Sigma_q$ for the  $\equiv_q^{\sf MSO}$ equivalence classes 
of $m$-colored trees,
namely one has
\begin{eqnarray*}
{[}T_i ]_q &=&
\{ \bullet_j\}\quad \text{if } T_i = \{\bullet_j\}; \ \text {otherwise}\\
{[}T_i ]_q &=&\bigcup \Big\{ \bullet_r\Big/\sum_{j=1}^k \gamma_j\star [T_j ]_q\ : \ \gamma_j \in \{1,\ldots, N-1, (\ge N)\}, \ 
T_i \equiv_q^{\sf MSO} \bullet_r\Big/ \sum_{j=1}^k \gamma_j\star T_j \Big\}.
\end{eqnarray*}
To test the last condition (concerning $\equiv_q^{\sf MSO}$) one replaces any $\gamma_i =(\ge N)$ by $N$, 
so one is deciding $\equiv_q^{\sf MSO}$ between two trees.

\item
One can effectively find the dependency diagraph of $\Sigma_q$ (immediate from the previous step).

\item
One can effectively find the periodicity parameters (as defined in \cite{BBY2006}) of the spectra of the $[T_i]_q$.
\end{thlist}

%
\begin{question} 
One questions stands out, namely can one find an explicit bound (in terms of known
functions, like exponentiation) for the value $n = N+1$ for which the WHILE loop halts? This would give an upper 
bound on the height of a set of smallest possible representatives of the $\equiv_q^{\sf MSO}$ classes of trees.
\end{question}

 In conclusion, 
a strong point in favor of Compton's approach, besides its simplicity in proving the 
 foundational result on the spectra of MSO classes of trees, is that it also gives
a defining system for the generating functions, and hence offers the possibility of
understanding the periodicity parameters described in Definition \ref{defn per param}
and the asymptotics for the growth of MSO classes. Using Compton's equations
we have carried out a detailed study \cite{BBY2010} of MSO classes $\cT$ of trees whose
generating function $T(x)$  has radius of convergence
$\rho=1$. One conclusion obtained was that if the class of forests $\partial \cT$ is closed
under addition, and under extraction of trees (thus forming an additive number system as
described in \cite{Burris2001}), then $\cT$ has a MSO 0--1 law.
 
 %

 \end{document}

%% file: symb1.tex
\newcommand{\raStar}{\rightarrow^\star}

\newcommand{\ldegree}{{\rm ldegree}}

\newcommand{\rt}{{\sf rt}}

\newcommand{\Set}{{\sf Set}}
\newcommand{\MSet}{{\sf MSet}}
\newcommand{\Seq}{{\sf Seq}}
\newcommand{\DCycle}{{\sf DCycle}}

\newcommand{\Cycle}{{\sf Cycle}}

\newcommand{\Su}{{\rm Su}}

\newcommand{\nequiv}{\not{\hspace*{-0.8ex}\equiv}}
\newsymbol \ndiv   232D
\newsymbol \nmodels  2332

\newcommand{\FORESTS}{{\rm FORESTS}}

\newcommand{\TREES}{{\rm TREES}}

\newcommand{\MSO}{{\rm MSO}}

\newcommand{\dom}{{\sf Dom}}
\newcommand{\Sdom}{\Gamma{\sf Dom}}
\newcommand{\Spec}{{\sf Spec}}
\newcommand{\Primes}{{\sf Primes}}
\newcommand{\EvenSquares}{{\sf PosEven}}

\newcommand{\bgam}{{\pmb{\Gamma}}}
\newcommand{\gam}{{\Gamma}}

\newcommand{\bA}{{\bf A}}
\newcommand{\bB}{{\bf B}}

\newcommand{\bF}{{\bf F}}
\newcommand{\bG}{{\bf G}}
\newcommand{\bH}{{\bf H}}

\newcommand{\bL}{{\bf L}}

\newcommand{\bR}{{\bf R}}
\newcommand{\bS}{{\bf S}}
\newcommand{\bT}{{\bf T}}
\newcommand{\bU}{{\bf U}}

\newcommand{\bY}{{\bf Y}}

\newcommand{\bb}{{\bf b}}

\newcommand{\bu}{{\bf u}}
\newcommand{\bv}{{\bf v}}

\newcommand{\by}{{\bf y}}
\newcommand{\bz}{{\bf z}}

\newcommand{\bzero}{{\mathbf 0}}

\newcommand{\bempty}{{\pmb{\O}}}

\newcommand{\bGamma}{{\mathbf{\Gamma}}}
\newcommand{\bDelta}{{\mathbf{\Delta}}}

\newcommand{\bOmega}{{\mathbf{\Omega}}}
\newcommand{\bTheta}{{\mathbf{\Theta}}}

\newcommand{\bbC}{{\mathbb C}}
\newcommand{\bbF}{{\mathbb F}}

\newcommand{\bbN}{{\mathbb N}}
\newcommand{\bbP}{{\mathbb P}}

\newcommand{\bbR}{{\mathbb R}}
\newcommand{\bbS}{{\mathbb S}}

\newcommand{\bbZ}{{\mathbb Z}}

\newcommand{\fQ}{{\mathfrak Q}}

\newcommand{\fc}{{\mathfrak c}}

\newcommand{\fm}{{\mathfrak m}}
\newcommand{\fn}{{\mathfrak n}}
\newcommand{\fp}{{\mathfrak p}}
\newcommand{\fq}{{\mathfrak q}}

\newcommand{\bfm}{{\pmb{\mathfrak m}}}

\newcommand{\bfq}{{\pmb{\mathfrak q}}}

\newcommand{\cA}{{\mathcal A}}
\newcommand{\cB}{{\mathcal B}}
\newcommand{\cC}{{\mathcal C}}

\newcommand{\cF}{{\mathcal F}}

\newcommand{\cL}{{\mathcal L}}

\newcommand{\cN}{{\mathcal N}}
\newcommand{\cO}{{\mathcal O}}

\newcommand{\cR}{{\mathcal R}}

\newcommand{\cT}{{\mathcal T}}
\newcommand{\cU}{{\mathcal U}}

\newcommand{\bcA}{{\pmb{\mathcal A}}}
\newcommand{\bcB}{{\pmb{\mathcal B}}}

\newtheorem{theorem}{Theorem}
\newtheorem{proposition}[theorem]{Proposition}
\newtheorem{lemma}[theorem]{Lemma}
\newtheorem{corollary}[theorem]{Corollary}

\newtheorem{definition}[theorem]{Definition}
\newtheorem{example}[theorem]{Example}

\newtheorem{question}{Question}

\newtheorem{remark}[theorem]{Remark}

 \newcounter{thlistctr}
 \newenvironment{thlist}{\
 \begin{list}%
 {\alph{thlistctr}}%
 {\setlength{\labelwidth}{2ex}%
 \setlength{\labelsep}{1ex}%
 \setlength{\leftmargin}{6ex}%
 \usecounter{thlistctr}}}%
 {\end{list}}

\newcount\minute        
\newcount\hour          
\newcount\hourMins  
\def\now%
{
  \minute=\time    
  \hour=\time \divide \hour by 60 
  \hourMins=\hour \multiply\hourMins by 60
  \advance\minute by -\hourMins 
  \zeroPadTwo{\the\hour}:\zeroPadTwo{\the\minute}%
}
\def\timestamp%
{
  \today\ \now
}
\def\zeroPadTwo#1%
{
  \ifnum #1<10 0\fi    
  #1
}

%% file: SingleFile.bbl
\begin{thebibliography}{99}
%
\bibitem{BBY2006}
Jason P.~Bell, Stanley N.~Burris,  and Karen A.~Yeats,
{\em Counting Rooted Trees: The Universal Law $t(n) \sim C \cdot \rho^{-n} \cdot  n^{ -3/2}$}.
The Electron. J. Combin. {\bf 13} (2006), R63 [64pp.]

\bibitem{BBY2009}
------,
{\em Characteristic Points of Recursive Systems}. Preprint, May 2009, 39 pp.

\bibitem{BBY2010}
------,
{\em Monadic Second Order Classes of Trees of Radius 1}. (In Preparation.)

\bibitem{Berstel1971}
J.~Berstel, Sur les p\^{o}les et le quotient de Hadamard de s\'{e}ries $n$-rationnelles. 
\emph{Comptes-Rendus de lÕAcad«emie des Sciences}, {\bf 272}, S\'{e}rie A (1971), 1079--1081.

\bibitem{Buchi1960}
J.~Richard B\"uchi, 
{\em Weak second-order arithmetic and finite automata}.  
Z. Math. Logik Grundlagen Math. {\bf  6}  1960, 66--92.

\bibitem{Burris2001}
Stanley N.~Burris, 
{\em Logical Limit Laws and Number Theoretic Density}. 
Mathematical Surveys and Monographs, Vol. {\bf 86}, Amer. Math. Soc., 2001.

\bibitem{Cayley1857}
A.~Cayley, 
{\em On the theory of the analytical forms called trees}. 
Phil. Magazine {\bf 13} (1857), 172--176. 


\bibitem{Co2}
Kevin J.~Compton,
{\em A logical approach to asymptotic combinatorics.~II. Monadic second-order properties.}
J.~Combin.~Theory, Ser.~A {\bf 50} (1989), 110--131.

\bibitem{Compton2009}
Kevin Compton, {\em Private communication}, July, 2009.

\bibitem{CoHe1990}
Kevin J.~Compton and C.~Ward Henson,
{\em A uniform method for proving lower bounds on the computational complexity of logical theories.}
Annals of Pure and Applied Logic {\bf 48} (1990), 1-- 79.

\bibitem{DoSh2007}
Mor Doron  and Saharon Shelah,
{\em Relational structures constructible by quantifier free definable operations}.
Journal of Symbolic Logic,
{\bf 72} (2007), 
1283--1298.

\bibitem{DFL1997}
Arnaud Durand, Ronald Fagin and Bernd Loescher, 
{\em Spectra with only unary function symbols.} 
Proceedings of the 1997 Annual Conference of the European Association
for Computer Science Logic (CSLÕ97). 
[The paper can be found at http://www.almaden.ibm.com/cs/people/fagin/]

\bibitem{DJMM2009} 
Arnaud Durand, N.D.~Jones, J.A.~Makowsky, and M.~More,
{Fifty years of the spectrum problem.} (Preprint, July, 2009).

\bibitem{FiMa2004}
E. Fischer and J.A.~Makowsky, 
{\em On spectra of sentences of monadic second order logic with counting.} 
J. Symbolic Logic {\bf 69} (2004), no. 3, 617--640.

\bibitem{FlSe2009}
Philippe Flajolet and Robert Sedgewick,
 {\em Analytic Combinatorics}.
Cambridge University Press, 2009.

\bibitem{GuSh2003}
Yuri Gurevich and Saharon Shelah,
{\em Spectra of monadic second-order formulas with one unary function}.
18th Annual IEEE Symposium on Logic in Computer Science, 
June 22--25, 2003, Ottawa, Canada.

\bibitem{Or1987}
James M. Ortega, 
{\em Matrix Theory. A Second Course.}
Plenum Press, 1987.

\bibitem{PoRe1987}
G. P\'olya and R.C. Read, 
{\em Combinatorial enumeration of groups, graphs and chemical compounds}. 
Springer Verlag, New York, 1987.

\bibitem{Shelah2004}
Saharon Shelah,
{\em Spectra of monadic second order sentences}.
Scientiae Mathematicae Japonicae,
 {\bf 59}, No. 2,
(2004), 351--355.

\bibitem{Scholz1952}
Heinrich Scholz, 
{\em  Ein ungel\"ostes Problem in der Symbolischen Logik.}  
Journal of  Symbolic Logic, \textbf{17}, No. 2 (1952), p. 160.

\bibitem{Soittola1976}
M. Soittola, 
{\em Positive rational sequences}. 
Theoretical Computer Science \textbf{2} (1976), 317-322.

\bibitem{Stockmeyer1987}
Larry Stockmeyer, 
{\em Classifying the computational complexity of problems.} 
Journal of  Symbolic Logic, \textbf{52}, No. 1 (1987), 1--43.

\bibitem{Wilf1994}
Herbert S. Wilf, {\em Generatingfunctionology.} Academic Press, 1994.

\bibitem{Woods1997}
Alan R. Woods, 
{\em Coloring rules for finite trees, probabilities of monadic second-order sentences}. 
Random Structures Algorithms {\bf 10} (1997), 453--485.
\end{thebibliography}
